\documentclass[leqno,11pt,a4paper]{article}
\usepackage[left=2.0cm, top=2.2cm,bottom=2.5cm,right=2.0cm]{geometry}
\usepackage{amsmath,amssymb,amsthm,mathrsfs,calc,graphicx,color}
\usepackage[british]{babel}
\numberwithin{equation}{section}

\newtheorem{theorem}{Theorem}[section]
\newtheorem{proposition}[theorem]{Proposition}
\newtheorem{lemma}[theorem]{Lemma}
\newtheorem{corollary}[theorem]{Corollary}
\newtheorem{remark}[theorem]{Remark}

\def\ba{\begin{array}}
\def\ea{\end{array}}
\def\bea{\begin{eqnarray} \label}
\def\eea{\end{eqnarray}}
\def\be{\begin{equation} \label}
\def\ee{\end{equation}}
\def\bit{\begin{itemize}}
\def\eit{\end{itemize}}
\def\ben{\begin{enumerate}}
\def\een{\end{enumerate}}

\def\lan{\langle}
\def\ran{\rangle}
\def\ka{\kappa}

\def\EE{\mathbb{E}}

\def\NN{\mathbb{N}}
\def\PP{\mathbb{P}}
\def\QQ{\mathbb{Q}}
\def\RR{\mathbb{R}}
\def\SS{\mathbb{S}}

\def\a{\alpha}

\def\g{\gamma}
\def\d{\delta}

\def\k{\kappa}

\def\r{\varrho}
\def\s{\sigma}

\def\o{\omega}
\def\G{\Gamma}

\def\L{\Lambda}

\def\O{\Omega}
\DeclareMathOperator{\ind}{\mathbf{1}}

\def\bC{\mathbf{C}}

\def\bM{\mathsf{M}}

\def\cB{\mathcal{B}}

\def\cF{\mathcal{F}}

\def\cH{\mathcal{H}}

\def\cR{\mathcal{R}}
\def\cS{\mathcal{S}}

\def\dint{\textup{d}}

\def\var{\textup{var}}
\def\cov{\textup{cov}}
\def\inter{\textup{int}}

\def\bd{\textup{bd}}

\begin{document}

\title{\bfseries Intersection and proximity of processes of flats}

\author{Daniel Hug\footnotemark[1], Christoph Th\"ale\footnotemark[2] and Wolfgang Weil\footnotemark[3]}

\date{}
\renewcommand{\thefootnote}{\fnsymbol{footnote}}
\footnotetext[1]{Karlsruhe Institute of Technology, Department of Mathematics,  D-76128 Germany. E-mail: daniel.hug@kit.edu}

\footnotetext[2]{Ruhr University Bochum, Faculty of Mathematics, D-44780 Bochum, Germany. E-mail: christoph.thaele@rub.de}

\footnotetext[3]{Karlsruhe Institute of Technology, Department of Mathematics,  D-76128 Germany. E-mail: wolfgang.weil@kit.edu}

\maketitle

\begin{abstract}
Weakly stationary random processes of $k$-dimensional affine subspaces (flats) in $\RR^n$ are considered. If $2k\geq n$, then intersection processes are investigated, while in the complementary case $2k<n$ a  proximity process is introduced. The intensity measures of these processes are described in terms of parameters of the underlying $k$-flat process. By a translation into geometric parameters of associated zonoids and by means of integral transformations, several new uniqueness and stability results for these processes of flats are derived. They rely on a combination of known and novel estimates for area measures of zonoids, which are also developed in the paper. Finally, an asymptotic second-order analysis as well as central and non-central limit theorems for length-power direction functionals of proximity processes derived from stationary Poisson $k$-flat process complement earlier works for intersection processes.
\end{abstract}

\section{Introduction}

The aim of this work is to discuss the interplay and duality of the notions `intersection' and `proximity' for locally finite random systems $X$ of $k$-dimensional affine subspaces (called {\it $k$-flats}, for short) in $n$-dimensional Euclidean spaces $\RR^n$, $n\ge 2$. More precisely, we consider weakly stationary point processes $X$ on the space $A(n,k)$ of $k$-flats in $\RR^n$ for $k\in\{0,\ldots ,n-1\}$. Here we call $X$ weakly stationary if its intensity measure is translation invariant, see Section \ref{sec:preliminaries} for precise definitions. Stationary $k$-flat processes are one of the basic models in stochastic geometry and have been studied intensively, for example, in \cite{BL,GHR96,HugLastWeil,Matheron,Mecke91,Schneider99,ST,spodarev1,Thomas84,Weil87}, see also \cite{SW} and the references therein. While stationarity implies weak stationarity, the converse is not true in general. However, for Poisson processes the two concepts are equivalent. 

If $k\geq n/2$, any two $k$-flats of $X$, which are in general position, intersect in a $(2k-n)$-dimensional subspace.  This gives rise to the {\it intersection process} $X_{(2)}$ of order two of $X$. Under a suitable condition on $X$, the process $X_{(2)}$ is again weakly stationary, its 
intensity is called the {\it intersection density} $X_{(2)}$ of $X$. Even if $X$ is a Poisson process, $X_{(2)}$ is not a Poisson process any more. For a stationary Poisson process $X$, the  intensity measure of the intersection process was described in \cite[Theorem 4.4.9]{SW} and, in the case of a stationary Poisson hyperplane process $X$, an upper bound for the intersection density  was given in \cite[Theorem 4.6.5]{SW}, based on the method of associated zonoids. Analogously, intersection processes $X_{(r)}$ of higher order $r\in \{2,\ldots,n\}$ and their intensities $\gamma_{(r)}(X)$ were considered in \cite[Theorem 4.4.8]{SW}, where $X_{(r)}$ arises from the intersections of any selection of $r$ hyperplanes in general position of a stationary Poisson hyperplane process $X$.

If $X$ is a weakly stationary process of $k$-flats with $1\le k< n/2$, a somehow dual situation arises. Again under a suitable condition on $X$, any pair $(E,F)$ of disjoint $k$-flats of $X$ in general position has now a positive distance $d(E,F)>0$, which is attained in uniquely determined points $x_E\in E$ and $x_F\in F$. They give rise to a random line segment $s(E,F)=\overline{x_Ex_F}$, the perpendicular of $E$ and $F$. If we only consider pairs $E,F\in X$ with distance $d(E,F)\le\delta$, for a fixed distance threshold $\delta>0$, then the segment-midpoints $m(E,F)=(x_E+x_F)/2$ 
build a weakly stationary point process in $\RR^n$, the intensity of which is denoted by $\pi(X,\d)$ (without the distance condition, the midpoints may even lie dense in $\RR^n$). The quantity $\pi(X,\d)$ is called \textit{proximity} of $X$ with distance threshold $\d$. For a stationary Poisson process $X$ and $\delta = 1$, it has been introduced and studied by Schneider \cite{Schneider99} (see also \cite[Theorem 4.4.10]{SW}).  It is known that if $X$ is a stationary Poisson line process in $\RR^n$ with fixed intensity, $\pi(X,\delta)$ attains its maximal value if and only if $X$ is isotropic (compare \cite[Theorem 4.6.6]{SW}).

The mentioned duality is based on the fact that stationary Poisson processes $X$ of $k$-flats are in one-to-one correspondence with pairs $(\g,\QQ)$, where $\g>0$ is the intensity and $\QQ$ is the directional distribution of $X$, a probability measure on the Grassmannian $G(n,k)$ of $k$-dimensional linear subspaces of $\RR^n$. This fact allows us to associate with $X$ a dual Poisson process $X^\perp$ of $(n-k)$-flats, which has the same intensity $\g$ and whose directional distribution $\QQ^\perp$ is given by the image of $\QQ$ under the orthogonal complement map $L\mapsto L^\perp$. Then, writing $\k_{n-2k}$ for the volume of the $(n-2k)$-dimensional unit ball and $\pi(X)$ for $\pi(X,1)$, we have the relation
\begin{equation*}
\pi(X)=\k_{n-2k}\,\gamma_{(2)}(X^\perp)
\end{equation*} 
according to the main result in \cite{Schneider99} or the remark after Theorem 4.4.10 in \cite{SW} (both cited results contain a factor $1/2$ which has to be removed). This clearly expresses the relation between the proximity $\pi(X)$ and the second intersection density $\g_{(2)}(X^\perp)$.

\medspace

In this paper, we consider the {\it proximity process} built by the segments $s(E,F)$, described above, and we study intersection processes and proximity processes as well as their interplay in more detail and in greater generality. In particular, we work in the framework of weakly stationary $k$-flat processes. This is also the setting which allows us to construct examples (in the appendix) of $k$-flat processes which are not Poisson, but have moment properties similar to that of Poisson processes. To enhance the readability of our text and to make the paper self-contained, Section \ref{sec:preliminaries} contains the basic notions and results from convex and stochastic geometry which are needed in the following. In particular, we recall the notion of an associated zonoid and provide a brief description of mixed volumes and area measures. In Section \ref{sec:Stability}, we prove a stability result for the area measures of zonoids which cannot be found in the existing literature. Then, in Section \ref{sec:Intersection}, we discuss intersection processes and present some extensions and generalizations of results in \cite{SW}, in particular a stability estimate for the directional distribution of a flat process $X$ if the directional distribution of a lower-dimensional intersection process is given. In Section \ref{sec:proximity}, we introduce the proximity process $\Phi$ associated with a weakly stationary $k_1$-flat process $X_1$ and a weakly stationary $k_2$-flat process $X_2$, if $X_1$ and $X_2$ are stochastically independent and $k_1+k_2<n$. This is the random process of all line segments $s(E,F)$ perpendicular to flats $E\in X_1$ and $ F\in X_2$ in general position and such that the length of $s(E,F)$ is bounded from above by a prescribed distance threshold. We describe its intensity measure in terms of the intensities and the directional distributions of the original processes $X_1$ and $X_2$. We also consider the proximity process $\Phi$ of a single $k$-flat process $X$ with $1\le k<n/2$.  For a Poisson line process $X$, we show that the intensity measure of $\Phi$ determines $X$ uniquely (in distribution), whereas for $k>1$ such a uniqueness result does not hold any more. In addition, for line processes a stability result is obtained. 

In the background of some of our results in Sections \ref{sec:Intersection} and \ref{sec:proximity} is a translation of probabilistic problems to geometric ones via integral transformations, such as the cosine transform. More precisely, we associate a zonoid with a weakly stationary $k$-flat process in $\RR^n$ ($k\in\{1,n-1\}$) and deduce that the intensity and the directional distribution of an intersection process (if $k=n-1$), or the proximity process (if $k=1$), can be re-written as an intrinsic volume and a suitably normalized area measure of this zonoid. In \cite{SW}, this translation together with isoperimetric inequalities for intrinsic volumes has been used to study extremal problems. While this approach is well established by now, stability estimates for lower-order area measure have not been used in stochastic geometry so far. Other stability estimates or stability results for top-order area measures have been developed and applied to problems in stochastic geometry in \cite{BorHug,HugSchneiderTessellations07,HS2010}, for example. It is one purpose of the present paper to fill this gap in the case of weakly stationary $k$-flat processes. 

In the final section, we perform an asymptotic second-order analysis of proximity processes $\Phi$ derived from a stationary Poisson $k$-flat process, $1\le k<n/2$, and establish central and non-central limit theorems. In the dual situation, namely for intersection processes, corresponding results have recently been obtained in \cite{LPST,RS,STScaling}. First, we derive a univariate as well as a multivariate central limit theorem for certain length-power direction functionals of $\Phi$. Then, we investigate the limiting behavior of certain order statistics and derive, for example, a Weibull limit theorem for the shortest segment of the proximity process. This generalizes parts of the theory developed in \cite{ST}. These results are based on the Wiener-It\^o chaos expansion of U-statistics of Poisson processes from \cite{RS}, relying itself on the Fock space representation and on covariance identities of general Poisson functionals developed in \cite{LastPenrose}. We also use recent cumulant formulae for U-statistics of Poisson point processes from \cite{LPST} together with the classical method of cumulants to prove our central limit theorems.

The paper is completed by an appendix in which an explicit construction of a class of weakly stationary $k$-flat processes $X$ is provided, which are of type $(S_r)$ and are not Poisson processes. Here, the $(S_r)$-condition is the crucial assumption under which some of the results in Sections \ref{sec:Intersection} and \ref{sec:proximity} are derived. It requires that the $r$th factorial moment measure of $X$ is equal to the $r$-fold product of the intensity measure of $X$, a condition which is satisfied by Poisson processes for all $r\ge 2$. Since we have found in the literature a construction of such processes only for $k=0$ and $r=2$, we include the material here to motivate our more general framework and to make the paper reasonably self-contained.

\section{Preliminaries}\label{sec:preliminaries}

\paragraph{General notation.} Throughout the following, we work in an $n$-dimensional Euclidean space $\RR^n$, $n\geq 2$, with 
scalar product $\langle\,\cdot\,,\,\cdot\,\rangle$ and norm $\|\,\cdot\,\|$. Hence, $\|x-y\|$ is the Euclidean distance of two points $x,y\in\RR^n$ and  $d(X,Y):=\inf\{\|x-y\|:x\in X,y\in Y\}$ is the distance of two sets $X,Y\subset\RR^n$. If $X=\{x\}$ is a singleton, we write $d(x,Y)$ for the distance of $\{x\}$ and $Y$. For a linear or affine subspace $E$ of $\RR^n$, we use $\ell_E$ for the Lebesgue measure on $E$ and $\s_E$ for the spherical Lebesgue measure on the unit sphere $\SS_E :=\{x\in L(E):\|x\|=1\}$ in $L(E)$, where $L(E)$ is the linear subspace parallel to $E$. For convenience we put $\SS^{n-1}:=\SS_{\RR^n}$, $\s_k:=\s_{\RR^k}$, $\ell_k:=\ell_{\RR^k}$ for $k\in\{1,\ldots,n\}$ and $\ell:=\ell_n$. The unit ball in $\RR^k$ is denoted by $B^k$ and we put $\k_k:=\ell_k(B^k)$ and $\o_k:=\s_k(\SS^{k-1})=k\k_k$.  For $s\geq 0$, let $\cH^s$ be the $s$-dimensional Hausdorff measure. In particular, $\cH^0$ is the counting measure.

\paragraph{Grassmannians.} For $k\in\{0,\ldots,n\}$, we denote by $G(n,k)$ and $A(n,k)$ the spaces of $k$-dimensional linear and affine subspaces of $\RR^n$, respectively, both supplied with their natural topologies (see \cite[Section 13.2]{SW}, for example). The elements of $A(n,k)$ are also called {\it $k$-flats}. By $\nu_k$ we denote the unique Haar probability measure on $G(n,k)$ and by $\mu_k$ the invariant measure on $A(n,k)$, normalized as in \cite{SW}.


Recall that two subspaces $L\in G(n,k_1)$ and $M\in G(n,k_2)$ are said to be in {\it general position} if the linear hull of $L\cup M$ has dimension $k_1+k_2$ if $k_1+k_2<n$ or if $L\cap M$ has dimension $k_1+k_2-n$ if $k_1+k_2\ge n$. We also say that two flats $E\in A(n,k_1)$ and $F\in A(n,k_2)$ are in general position if this is the case for $L(E)$ and $L(F)$. 
In particular, if $k\ge n/2$ and $E,F\in A(n,k)$ are in general position, then $E\cap F$ is a $(2k-n)$-flat. If $k< n/2$ and $E,F\in A(n,k)$ are in general position, then there are two uniquely determined points $x_E\in E$ and $x_F\in F$ such that $d(E,F)=\|x_E-x_F\|$ and we call $s(E,F) := \overline{x_Ex_F}$ the {\it orthogonal segment} and $m(E,F) :=(x_E+x_F)/2$ the {\it midpoint} of $E$ and $F$.

For linear subspaces $L_1,\ldots,L_r\subset\RR^n$, satisfying either $\dim(L_1)+\ldots+\dim(L_r)\leq n$ or $\dim(L_1)+\ldots+\dim(L_r)\geq (r-1)n$, we denote by $[L_1,\ldots,L_r]$ the {\it subspace determinant}  (see \cite[Section 14.1]{SW}). In particular, if 
$L\in G(n,k_1)$ and $M\in G(n,k_2)$ with  $k_1+k_2\leq n$, then $[L,M]$ is the $(k_1+k_2)$-volume of the parallelepiped spanned by $u_1,\ldots,u_{k_1},v_1,\ldots,v_{k_2}$, where $u_1,\ldots,u_{k_1}$ is an orthonormal basis in $L$ and $v_1,\ldots,v_{k_2}$ an orthonormal basis in $M$. If otherwise $k_1+k_2>n$, then $[L,M]=[L^\perp,M^\perp]$. 
If $E_1,\ldots,E_r$ are affine flats, we define $[E_1,\ldots,E_r]:=[L_1,\ldots,L_r]$, where $L_i=L(E_i)$ for $i=1,\ldots,r$.
In addition, if $k\geq 1$ and $u_1,\ldots,u_k\in\SS^{n-1}$ are unit vectors, we denote by $\nabla_k(u_1,\ldots,u_k)$ the $k$-volume of the parallelepiped spanned by $u_1,\ldots,u_k$. 
Using the subspace determinant, the notion of general position of flats $E,F\in A(n,k)$ can easily be generalized to more than two flats and to different dimensions. Namely, $E_1,\ldots,E_r \in \bigcup_{i=1}^{n-1}A(n,i)$ are in {\it general position} if and only if $[E_1,\ldots,E_r]> 0$ in either of the two cases considered above. In particular, if $E_1,\ldots,E_r$ are in general position with $\dim(E_i) =: k_i$  and if $q:=k_1+\ldots+k_r-(r-1)n\ge 0$, then the subspaces
$L(E_1)^\perp,\ldots,L(E_r)^\perp$ are linearly independent, and therefore
$$
\dim(E_1\cap\ldots\cap E_r)=n-\dim\big(L(E_1)^\perp+\ldots+L(E_r)^\perp\big)=n-(n-k_1)-\ldots-(n-k_r)=q\,.
$$

At several occasions, we will need the value of an integrated subspace determinant. For this, fix integers $0< r,s\leq n-1$ such that $n-r-s\geq 0$ and let $L\in G(n,r)$. Then
\begin{equation}\label{eq:Defcnk1k2}
c(n,r,s):=\int\limits_{G(n,s)}[L,M]\,\nu_{s}(\dint M) = \frac{\binom{n-r}{s}\k_{n-r}\k_{n-s}}{\binom{n}{s}\k_n\k_{n-r-s}}\,,
\end{equation}
independently of $L$, see \cite[Lemma 4.4]{HugSchneiderSchuster2008} or \cite[Corollary 4-5-5]{Matheron}.

\paragraph{Processes of flats.} For $k\in\{0,\ldots,n-1\}$, a {\it $k$-flat process} $X$ in $\RR^n$ is a simple point process on the space $A(n,k)$. Thus, $X$ is a measurable mapping from an underlying probability space $(\O,\cF,\PP)$ to the space ${\sf N}(A(n,k))$ of simple counting measures on $A(n,k)$ equipped with a natural $\s$-algebra (for more details, see, for example, \cite[Chapter 3.1]{SW}). Alternatively, ${\sf N}(A(n,k))$ can be viewed as the space of  countable subsets $\{E_i:i\in\NN\}$ of $A(n,k)$ satisfying $E_i\neq E_j$ for all $i\neq j$ and such that for any compact subset $B\subset\RR^n$ the set $\{i:E_i\cap B\neq\emptyset\}$ is finite. We will use both interpretations simultaneously. Hence, if $X$ is a $k$-flat process and $A$ is a Borel subset of $A(n,k)$, then $X(A)$ stands for the number of flats $E\in X$ satisfying $E\in A$. The intensity measure $\Theta$ of $X$ is defined by $\Theta(A) :=\EE X(A)$, for Borel sets $A\subset A(n,k)$. We always assume that the intensity measure of a $k$-flat process is locally finite, that is, $\Theta(\{E\in A(n,k):E\cap C\neq\emptyset\})<\infty$ for all compact sets $C\subset\RR^n$. We say that $X$ is \textit{weakly stationary} if $\Theta$ is a translation invariant measure on $A(n,k)$. In this case, Theorem 4.4.1 in \cite{SW} implies that there exists a constant $0\leq\g<\infty$ and a probability measure $\QQ$ on $G(n,k)$ such that 
\begin{equation}\label{eq:Intensity}
\int\limits_{A(n,k)} f(E)\,\Theta(\dint E) = \g \int\limits_{G(n,k)}\int\limits_{L^\perp}f(L+x)\,\ell_{L^\perp}(\dint x)\,\QQ(\dint L)\,,
\end{equation}
for all non-negative measurable functions $f$ on $A(n,k)$. If $\g\neq 0$ and $A\subset G(n,k)$ is a Borel set, it follows that
\begin{equation}\label{eq:GammaQQ}
\g=\k_{n-k}^{-1}\,\EE X(\{E:E\cap B^n\neq\emptyset\})\quad{\rm and}\quad\QQ(A)=\frac{\EE X(\{E:E\cap B^n\neq\emptyset,\,L(E)\in A\})}{\EE X(\{E:E\cap B^n\neq\emptyset\})}\,,
\end{equation}
where it is clear from the context that $E\in A(n,k)$. If $\g=0$, then \eqref{eq:Intensity} holds with an arbitrary and thus non-unique probability measure $\QQ$ on $G(n,k)$. The representations in \eqref{eq:GammaQQ} explain why $\g$ is called the {\it intensity} of $X$ and $\QQ$ its {\it directional distribution}. The process $X$ is said to be {\it stationary} if $X+z$ has the same distribution as $X$, for all $z\in\RR^n$. Clearly, a stationary $k$-flat process is weakly stationary, but in general the converse is not true. Moreover, we say that $X$ is {\it isotropic} if its distribution is invariant under rotations. In this case and if $\gamma>0$, $\QQ$ is the Haar probability measure $\nu_k$ on $G(n,k)$.

A special class of $k$-flat processes arises if we assume that $X(A)$ is Poisson distributed with mean $\Theta(A)$, for all Borel sets $A\subset A(n,k)$. By R\'enyi's theorem \cite[Theorem 6.5.12]{Cinlar}, this automatically implies that the random variables $X(A_1),\ldots,X(A_j)$ are independent for disjoint Borel sets $A_1,\ldots,A_j\subset A(n,k)$ and $j\in\NN$. In this case, we call $X$ a {\it Poisson $k$-flat process}. Such processes are particularly attractive since their distribution is uniquely determined by their intensity measure $\Theta$, cf.\ \cite[Theorem 3.2.1]{SW}. In other words, a weakly stationary Poisson $k$-flat process is uniquely determined in distribution by the intensity $\g$ and the directional distribution $\QQ$. Moreover, a weakly stationary Poisson $k$-flat process is stationary.

\paragraph{Factorial moment measures and processes of type $(S_r)$.} Let $X$ be a $k$-flat process in $\RR^n$ with intensity measure $\Theta=\EE X$ and $k\in\{0,\ldots,n-1\}$. For $m\in\NN$, the 
$m$-fold product measure $\Theta^m$ of $\Theta$ is a measure on the Borel sets of $A(n,k)^m$. In addition, we will use the {\it $m$th factorial moment measure} $\Theta^{(m)}$ of $X$, which is defined for measurable subsets $A\subset A(n,k)^m$ by 
\begin{equation}\label{eq:Campbell}
\Theta^{(m)}(A):=\EE\sum_{(E_1,\ldots,E_m)\in X_{\neq}^m} \ind\{(E_1,\ldots,E_m)\in A\}\,,
\end{equation}
where $X_{\neq}^m$ denotes the collection of all $m$-tuples of distinct $k$-flats of $X$, see \cite[Theorem 3.1.3]{SW}. If $X$ is a Poisson process, then the 
  (multivariate) \textit{Mecke formula} \cite[Corollary 3.2.3]{SW} implies that  $\Theta^{(m)}=\Theta^m$ (see \cite[Corollary 3.2.4]{SW}), for all integers $m\ge 2$. For a given 
integer $r\ge 2$ and a $k$-flat process $X$ we say that $X$ is {\it of type $(S_r)$}, if $\Theta^{(r)}=\Theta^r$. It is known that $X$  is a Poisson process if and only if $X$ is of type $(S_r)$ for all 
$r\ge 2$ (see \cite[Theorem 4.4]{Heinrich}). Examples of (weakly stationary) $k$-flat processes $X$, which are not Poisson but of type $(S_r)$ for some $r\ge 2$, are constructed in the appendix.

\paragraph{Segment processes.} By $\cS$ we denote the space of non-degenerate line segments in $\RR^n$. We supply $\cS$ with the Borel $\sigma$-algebra generated by the Hausdorff metric. A \textit{segment process} $\Phi$ is a simple point process on the measurable space $\cS$. As in the case of $k$-flat processes, it is convenient to identify the random counting measure $\Phi$ with the random collection of line segments induced by it. 
We notice that a segment process is a point process of convex particles of dimension one as considered in Chapter 4 of \cite{SW}. In particular, $\Phi$ is {\it stationary}, if its distribution is invariant under translations, and  {\it isotropic} if it is invariant under rotations around the origin. We say that $\Phi$ is \textit{weakly stationary}, if its intensity measure is translation invariant. Clearly, a stationary segment process is weakly stationary, but in general not vice versa.

For a line segment $s\in\cS$ let us denote by $m(s)\in\RR^n$ its midpoint, by $d(s):=\ell_1(s)\in(0,\infty)$ its length and by $\phi(s)\in G(n,1)$ its direction, which is the $1$-dimensional linear subspace parallel to $s$. Then, $\cS$ can be identified with the product space $\RR^n\times(0,\infty)\times G(n,1)$. This transfer corresponds to the representation of a segment process $\Phi$ as a marked point process on $\RR^n$ with mark space $(0,\infty)\times G(n,1)$. For the intensity measure $\Lambda(\,\cdot\,):=\EE \Phi(\,\cdot\,)$ of $\Phi$ it will be convenient to replace $G(n,1)$ by $\SS^{n-1}$ and to regard $\Lambda$ as a measure on  $\RR^n\times(0,\infty)\times\SS^{n-1}$, which is symmetric (even) in its third component in that $\L(\,\cdot\,,\,\cdot\,,C)=\L(\,\cdot\,,\,\cdot\,,-C)$ for any Borel set $C\subset\SS^{n-1}$. This can be achieved by identifying each line in $G(n,1)$ with one of its generating unit vectors in $\SS^{n-1}$ and then by symmetrizing the resulting measure (see the subsequent discussion for associated  zonoids). If $\Phi$ is weakly stationary and $\Lambda\neq 0$, $\Lambda$ is translation invariant in the first component, the \textit{intensity} $N$ is non-zero and $N$ and the \textit{directional distribution} $\cR$ of $\Phi$ are given by
\begin{equation}\label{eq:NvR}
N:=\kappa_n^{-1}\,\L(B^n\times(0,\infty)\times\SS^{n-1})\qquad{\rm and}\qquad \cR(C):=\frac{\L(B^n\times(0,\infty)\times C)}{\L(B^n\times(0,\infty)\times\SS^{n-1})}\,,
\end{equation}
where $C\subset\SS^{n-1}$ is a Borel set.

\paragraph{Associated zonoids.}
Let $K\subset\RR^n$ be a convex body (a non-empty, compact, convex set). We refer to \cite{Schneider93} for notions from convex geometry and their basic properties, which we use here and in the following.  The support function $h(K,\,\cdot\,):\RR^n\to\RR$ of $K$ is defined by 
$h(K,u):=\max\{\langle x,u\rangle:x\in K\}$. Any sublinear (i.e., subadditive and positively homogeneous of degree 1) real function on $\RR^n$ 
 is the support function of a uniquely determined convex body (see \cite[Theorem 1.7.1]{Schneider93}). Hence, 
for a finite and even Borel measure $\mu$ on $\SS^{n-1}$, there exists a convex body $Z_\mu$ with support function $h(Z_\mu,\,\cdot\,)$ given by 
\begin{equation}\label{support}
h(Z_\mu,x)=\frac{1}{2}\int\limits_{\SS^{n-1}}|\langle v,x\rangle|\, \mu(\dint v)\,,\qquad x\in\RR^n\,.
\end{equation}
Recall that a zonotope is the Minkowski sum of finitely many line segments and that a zonoid is a centrally symmetric convex body, which can be approximated (in the sense of the Hausdorff distance) by zonotopes. 
Since $x\mapsto |\langle v,x\rangle|$ is the support function of the segment $\overline{(-v)v}$, the convex body $Z_\mu$ is a zonoid. We call it the \textit{zonoid associated with $\mu$}. In this context, $\mu$ is called the \textit{generating measure} of $Z_\mu$.

In the following, we will often replace a Borel measure on $G(n,1)$ or on $G(n,n-1)$ by an even measure on $\SS^{n-1}$. To describe this in more detail, first let $\QQ$ 
be a Borel measure on $G(n,1)$, and let $C\subset \SS^{n-1}$ be a Borel set. Then 
\begin{equation}\label{symmeasure}
\tilde{\QQ}(C):=\frac{1}{2}\int_{G(n,1)}\mathcal{H}^0(C\cap U)\,\QQ(\dint U)=\frac{1}{2}\int_{G(n,1)}\int_U\ind_C(u)\,\mathcal{H}^0(\dint u)\, \QQ(\dint U)
\end{equation}
defines an even Borel meausure on $\SS^{n-1}$. Let $L(u)\in  G(n,1)$ be the subspace spanned by  $u\in \SS^{n-1}$.  Then, for any measurable function $f:G(n,1)\to[0,\infty)$, 
$$
\int_{G(n,1)}f(U)\, \QQ(\dint U)=\int_{\SS^{n-1}}f(L(u))\,\tilde{\QQ}(\dint u).
$$
Moreover, if $h:\SS^{n-1}\to[0,\infty)$ is measurable and even, and $\tilde h(U):=h(u)=h(-u)$ for $u\in U\cap \SS^{n-1}$ and $U\in G(n,1)$, then 
$$
\int_{\SS^{n-1}}h(u)\,\tilde{\QQ}(\dint u)=\int_{G(n,1)}\tilde{h}(U)\, \QQ(\dint U).
$$
Alternatively, if $T_+,T_-:G(n,1)\to\SS^{n-1}$ are  Borel measurable maps such that $\{T_+(U),T_-(U)\}=U\cap \SS^{n-1}$ (obtained, e.g., by lexicographic ordering), then $\tilde{\QQ}=\frac{1}{2}(T_+\QQ+T_-\QQ)$. With these remarks in mind, starting from a measure $\QQ$ on $G(n,1)$, we write again $\QQ$ for the even measure $\tilde{\QQ}$ on $\SS^{n-1}$. Similarly, if $\QQ$ is a given Borel measure on $G(n,n-1)$, we apply the preceding definitions to $\QQ^\perp$ (the measure obtained from $\QQ$ as the image measure of the orthogonal complement map), and thus again arrive at an even Borel measure $\hat\QQ$ on $\SS^{n-1}$, that is,
\begin{equation}\label{symhat}
\hat\QQ(C):=\frac{1}{2}\int_{G(n,n-1)}\mathcal{H}^0(C\cap U^\perp)\,\QQ(\dint U)=\frac{1}{2}\int_{G(n,n-1)}\int_{U^\perp}\ind_C(u)\,\mathcal{H}^0(\dint u)\, \QQ(\dint U),
\end{equation}
for a Borel set $C\subset \SS^{n-1}$. 
With the same notation as before, we then have
$$
\int_{G(n,n-1)}f(U)\, \QQ(\dint U)=\int_{\SS^{n-1}}f(u^\perp)\,\hat{\QQ}(\dint u).
$$
and
$$
\int_{\SS^{n-1}}h(u)\,\hat{\QQ}(\dint u)=\int_{G(n,n-1)}\tilde{h}(U^\perp)\, \QQ(\dint U).
$$

Let $X$ be a weakly stationary process of $k$-flats in $\RR^n$ with intensity $\g>0$ and directional distribution $\QQ$ and suppose that $k\in\{1,n-1\}$ (these two cases are sufficient for the applications below).  In this situation, the associated zonoids $Z_{\QQ}$ and $\gamma Z_{\QQ}=Z_{\gamma\QQ}$ are given by \eqref{support}, with the 
aforementioned identifications (see also the discussion in \cite[p.~131]{SW}). If $\gamma=0$, then $\gamma Z_{\QQ}=Z_{\gamma\QQ}=\{0\}$, even if $\QQ$ is not unique in this case.

\paragraph{Mixed and intrinsic volumes.}
Let $K_1,\ldots,K_m\subset\RR^n$, $m\geq 1$, be convex bodies and $a_1,\ldots,a_m\geq 0$ be real numbers. The volume of the Minkowski sum $a_1K_1+\ldots+a_mK_m$ can be expressed as a polynomial 
$$\ell(a_1K_1+\ldots+a_mK_m)=\sum_{i_1,\ldots,i_n=1}^m a_{i_1}\cdots a_{i_n}\,V(K_{i_1},\ldots,K_{i_n})$$ 
with uniquely determined symmetric coefficients $V(K_{i_1},\ldots,K_{i_n})$, which are the \textit{mixed volumes} of $K_1,\ldots,K_m$. For two convex bodies $K,L\subset\RR^n$ we shall use the special notation 
$$V(K[k],L[n-k]):=V(\underbrace{K,\ldots,K}_{k\ {\rm times}},\underbrace{L,\ldots,L}_{n-k\ {\rm times}})\,,\qquad k\in\{0,\ldots,n\}\,.$$ 
A similar notation is used if three convex bodies are involved. The \textit{intrinsic volumes} of a convex body $K\subset\RR^n$ are then given by 
\begin{equation}\label{eq:IntVolMixedVolRelation}
V_k(K):=\frac{\binom{n}{ k}}{\k_{n-k}}V(K[k],B^n[n-k])\,,\qquad k\in\{0,\ldots,n\}\,.
\end{equation} 
In particular, $V_0(K)={\bf 1}\{K\neq\emptyset\}$ is the Euler characteristic of $K$, $V_1(K)$ is a constant multiple of its mean width, $V_{n-1}(K)$ is half of the surface area of $K$ and $V_n(K)=\ell(K)$. For further background material we refer the reader to \cite{Schneider93} or Chapter 14 of \cite{SW}.

\paragraph{Area measures.}
We will have to deal with the so-called area measures of convex bodies. To introduce them, we need some notation. Fix a convex body $K$ and let $x\in\RR^n\setminus K$. By $\pi_K(x)$ we mean the unique nearest point of $x$ in $K$ (i.e., on its boundary) and by $n_K(x)$ we denote the unit vector pointing from $\pi_K(x)$ to $x$. For a Borel set $B\subset\SS^{n-1}$ and $r>0$, we consider the local parallel set $A(K,B,r):=\{x\in\RR^n:0<d(x,K)\leq r,\,n_K(x)\in B\}$ of points in $\RR^n\setminus K$ with distance at most $r$ from $K$ and such that their associated unit normals $n_K(x)$ are in $B$. Then, $A(K,B,r)$ is a Borel set and the volume $\ell(A(K,B,r))$ of $A(K,B,r)$ is a polynomial of degree $n$ in $r$, i.e., $$\ell(A(K,B,r))=\frac{1}{ n}\sum_{m=0}^{n-1}\binom{n}{ m}\,S_m(K,B)\,r^{n-m}\,.$$ The coefficients $S_0(K,B),\ldots,S_{n-1}(K,B)$ are Borel measures on $\SS^{n-1}$ in their second component. These are the \textit{area measures} associated with $K$, see Chapter 4.2 in \cite{Schneider93}. While $S_0(K,\cdot)$ is, 
independently of $K$, the spherical Lebesgue measure, the measure $S_m(K,\cdot)$, for $m\in\{1,\ldots,n-1\}$, uniquely determines $K$ up to translations among all convex bodies of dimension at least $m+1$, cf.\ \cite[Corollary 8.1.4]{Schneider93}. We also recall that the total measure of the $m$th area measure of $K$ is related to its $m$th intrinsic volume by
\begin{equation}\label{eq:VmSmRelation}
V_m(K)=\frac{\binom{n}{ m}}{ n\k_{n-m}}S_m(K,\SS^{n-1})\,,
\end{equation}
see \cite[Chapter 4.2, p.~216]{Schneider93}.

\section{A stability result}\label{sec:Stability}

We will later use a stability result for area measures of associated zonoids in terms of the corresponding generating measures. For area measures, stability results are known from \cite{HS,K,Schneider93} and for zonoids they have been investigated in  \cite{HS,HS2010,K,Schneider93}. Here, we combine these approaches and we also replace the total variation distance of measures by the bounded Lipschitz distance and by the Prohorov distance, which leads to stronger results. In view of our later application to Grassmannians $G(n,k)$, we consider in this paragraph a separable compact metric space $(S,\varrho)$ with Borel $\sigma$-field $\cB(S)$. The space of finite Borel measures on $S$ is denoted by $\bM(S)$ and for $0<\rho\leq R<\infty$, we let $\bM(\rho,R)$ be the subspace of all $\mu\in\bM(S)$ for which $\rho\le\mu(S)\le R$. We shall assume that all measures we are dealing with are finite (non-negative) Borel measures. For two such measures $\mu$ and $\nu$ the {\it Prohorov distance} (sometimes also called L\'evy-Prohorov distance) is defined by
$$
d_P(\mu,\nu) := \inf\{ \varepsilon >0 : \mu (A) \le \nu (A^\varepsilon) + \varepsilon,\, \nu (A) \le \mu (A^\varepsilon) + \varepsilon \text{ for all } A\in \cB(S)\},
$$
where $A^\varepsilon := \{x\in S:\varrho(x,A)<\varepsilon\}$ and $\varrho(x,A):=\inf\{\varrho(x,a):a\in A\}$ for $x\in S$ and $A\in\cB(S)$, see \cite[Chapter 8.3]{Bogachev} or \cite[Chapter 11.3]{Dud}.  For a function $f:S\to\RR$ we define the sup-norm, the Lipschitz-norm as well as the bounded Lipschitz-norm by
\begin{equation}\label{eq:SupUndLNorm}
\|f\|_\infty:=\sup\{|f(x)|:x\in S\},\qquad \|f\|_L:=\sup\left\{\frac{|f(x)-f(y)|}{\varrho(x,y)}:x,y\in S,x\neq y\right\}\,,
\end{equation}
and
$$
\|f\|_{BL}:=\|f\|_\infty+\|f\|_L\,,
$$
respectively. We put $\mathcal{F}_{BL}(S):=\{f:S\to \RR:\|f\|_{BL}\le 1\}$. 
The \textit{bounded Lipschitz distance} of $\mu$ and $\nu$ is defined by
$$
d_{BL}(\mu,\nu):=\sup\bigg\{\Big|\int\limits_S f\, \dint\mu-\int\limits_S f\,\dint\nu\Big|:\, f\in \mathcal{F}_{BL}(S)\bigg\}\,,
$$
see \cite[Chapter 8.3]{Bogachev}, \cite{Dudley1} or \cite[Chapter 11.3]{Dud}. We remark that both, the Prohorov distance as well as the bounded Lipschitz distance metrize the topology of weak convergence of measures on $S$.
The next lemma is known for probability measures (cf.\ \cite[Theorem 8.10.43]{Bogachev}).

\begin{lemma}\label{compare}
Let $0<\rho\leq R<\infty$ and $\mu,\nu\in \bM(\rho,R)$. Then there are constants  $c_1,c_2>0$ depending only on $\rho$ and $R$ such that 
$$
c_1\, d_P(\mu,\nu)^2\le d_{BL}(\mu,\nu)\le c_2\, d_P(\mu,\nu)\,.
$$
\end{lemma}

\begin{proof}
The right inequality is established in the course of the proof of Theorem 3.1 in \cite{HS}, all arguments  extend to 
the present more general setting of a compact metric space. To establish the left inequality, we use the abbreviation $\varepsilon:=d_{BL}(\mu,\nu)$. 
In the following, we can assume that $\varepsilon>0$. Putting $\mu_1:=\mu(S)$ and $\nu_1:=\nu(S)$, the definition of the bounded Lipschitz-distance with $f\equiv 1$ immediately implies that $|\mu_1-\nu_1|\le\varepsilon$. Then, for any $f:S\to\RR$ with $\|f\|_{BL}\le 1$, we get
\begin{align*}
\Big|\int\limits_S f\, \frac{\dint\mu}{\mu_1}-\int\limits_S f\,\frac{\dint\nu}{\nu_1}\Big|&\le \frac{1}{\mu_1}\Big|\int\limits_S f\, \dint\mu-\int\limits_Sf\,\dint\nu\Big|+\Big|\frac{\nu_1-\mu_1}{\nu_1\mu_1}\Big|\,\Big|\int\limits_S f\, \dint\nu\Big|\\
&\le c_3 d_{BL}(\mu,\nu)+c_4\varepsilon\\
&\le c_5 d_{BL}(\mu,\nu)\,,
\end{align*}
with suitable constants $c_3,c_4,c_5>0$. Hence,
\begin{equation}\label{blbound}
d_{BL}\Big(\frac{\mu}{\mu_1},\frac{\nu}{\nu_1}\Big)\le  c_5 d_{BL}(\mu,\nu)\,.
\end{equation}
Combining \eqref{blbound} with \cite[Theorem 8.10.43]{Bogachev} (see also \cite[Exercise 11.3.5(b)]{Dud}), we conclude that
\begin{equation}\label{firststep}
d_P\Big(\frac{\mu}{\mu_1},\frac{\nu}{\nu_1}\Big)\le\sqrt{\frac{3}{ 2}\,d_{BL}\Big(\frac{\mu}{\mu_1},\frac{\nu}{\nu_1}\Big)}< c_6\sqrt{d_{BL}(\mu,\nu)}
\end{equation}
with $c_6=\sqrt{2 c_5}$. 
Let $A\subset S$ be a Borel set. Recall that $\varepsilon =d_{BL}(\mu,\nu)\le 2R$ and define $\varepsilon':=c_6\sqrt{\varepsilon}$.  
Clearly, we have $\varepsilon\le\sqrt{2R}\sqrt{\varepsilon}$. Hence,
\eqref{firststep} implies that
\begin{align*}
\mu(A)&=\mu_1\frac{\mu(A)}{\mu_1}\le \mu_1\left(\frac{\nu(A^{\varepsilon'})}{\nu_1}+\varepsilon'\right)=\frac{\mu_1}{\nu_1}\nu(A^{\varepsilon'})
+\mu_1\varepsilon'\\
&\le \nu(A^{\varepsilon'})+\left|\frac{\mu_1}{\nu_1}-1\right|\nu(A^{\varepsilon'})+\mu_1\varepsilon'\\
&\le \nu(A^{\varepsilon'})+\left|\mu_1-\nu_1\right| +\mu_1\varepsilon' \\
&\le \nu(A^{\varepsilon'})+\varepsilon +\mu_1\varepsilon'\\
&\le \nu(A^{\varepsilon'})+c_7\varepsilon',
\end{align*}
with a suitable constant $c_7>0$. By symmetry, it follows that 
$$
d_P(\mu,\nu)\le c_8\sqrt{d_{BL}(\mu,\nu)}\,,
$$
with a  constant $c_8>0$, which yields the left inequality in the statement of the lemma.
\end{proof}


The next theorem deals with the particular case $S=\SS^{n-1}$, which is tailored towards our applications later in this paper. The metric $\varrho$ can be chosen (for instance) to be the intrinsic Riemannian metric or the restriction of the Euclidean metric to $\SS^{n-1}$. For $0<\rho\le R<\infty$, we say that a finite Borel measure $\mu$ on $\SS^{n-1}$ has {\it upper bound} $R$, if $\mu(\SS^{n-1})\le R$, and {\it lower bound} $\rho$, if 
$$\int\limits_{\SS^{n-1}}|\langle u,v\rangle|\, \mu(\dint v)\ge\rho\,,\qquad\text{for all}\ u\in\SS^{n-1}\,.$$ 
Let $\bM_{\rm e}(\SS^{n-1})$ be the space of all even and finite Borel measures on $\SS^{n-1}$ and $\bM_{\rm e}(\rho,R)$ the subspace consisting of all $\mu\in\bM_{\rm e}(\SS^{n-1})$ which have upper bound $R$ and lower bound $\rho$. Recall that a measure $\mu$ on $\SS^{n-1}$ is even if $\mu(A)=\mu(-A)$ for all Borel sets $A\subset\SS^{n-1}$, where $-A$ is the image of $A$ under the antipodal map. Clearly, $\bM_{\rm e}(\rho,R)\subset \bM(\rho,R)$ (with the choice  $S=\SS^{n-1}$), and if $\mu\in\bM_{\rm e}(\rho,R)$, then  $2Z_\mu$ lies in $\mathcal{K}(\rho,R)$, which is the class of convex bodies $K$ in $\RR^n$ satisfying 
$\rho B^n\subset K\subset R B^n$.  The \textit{Hausdorff distance} of two convex sets $K,M\subset\RR^n$ 
can be defined in terms of the support functions $h(K,\,\cdot\,), h(M,\,\cdot\,)$ of $K,M$, considered as functions on the unit sphere $\SS^{n-1}$, by $d_H(K,M):=\|h(K,\,\cdot\,)-h(M,\,\cdot\,)\|_{\infty}$, 
see \cite[Chapter 1.8]{Schneider93}. We can now state the main result of this section.

\begin{theorem}\label{stability} Let $m\in\{1,\dots, n-1\}$, and let $\mu,\nu\in\bM_{\rm e}(\rho,R)$ be even Borel measures on $\SS^{n-1}$. Then there is a constant $c=c(n,\rho,R)>0$, depending only on $n,\rho$ and $R$, such that
\begin{equation}\label{eq:GeneralStabilityEstimate}
d_{BL}(\mu,\nu) \le c\, d_{BL}\big(S_m(Z_\mu,\,\cdot\,),S_m(Z_\nu,\,\cdot\,)\big)^{2c(m,n)}\,,
\end{equation}
where $c(m,n):=\left(2^{m}(n+1)(n+4)\right)^{-1}$.
\end{theorem}

\begin{proof} We show that
\begin{equation}\label{stab1}
d_{BL}(\mu,\nu) \le c_1\, d_H(Z_\mu,Z_\nu)^{\alpha_1}
\end{equation}
and
\begin{equation}\label{stab2}
d_H(K,M) \le c_2\, d_{BL}\big(S_m(K,\cdot),S_m(M,\cdot)\big)^{\alpha_2}\,,
\end{equation}
for origin-symmetric convex bodies $K,M\in\mathcal{K}(\rho,R)$, appropriate constants $c_1,c_2>0$ and powers $\alpha_1:=(n+4)^{-1}$ and $\alpha_2:=((n+1)2^{m-1})^{-1}$. The desired inequality \eqref{eq:GeneralStabilityEstimate} then follows immediately with $c=c(n,\rho,R)=c_1c_2^{\alpha_1} $ and $2c(n,m)=\alpha_1\alpha_2$, since $Z_\mu,Z_\nu\in \mathcal{K}(\rho/2,R/2)$.

We first derive inequality \eqref{stab1} from \cite[Theorem 5.1]{HS}. For an easier comparison with that paper, we define for a signed Borel measure $\sigma$ on $\SS^{n-1}$ and a bounded measurable function $\Xi$ on $[-1,1]$, $$T_\Xi[\sigma](u):=\int\limits_{\SS^{n-1}}\Xi(\langle u,v\rangle)\,\sigma(\dint v)\,,\qquad u\in\SS^{n-1}\,.$$
In \cite{HS} an important r\^ole is played by the multipliers $a_{n,k}(\Xi)$ of $T_\Xi$, which are defined by $$a_{n,k}(\Xi):=\o_{n-1}\int\limits_{-1}^1(1-t^2)^{\frac{n-3}{2}}\,P_k^n(t)\,\Xi(t)\,\dint t\,,$$ where $P_k^n$ is the Legendre polynomial of dimension $n$ and degree $k$. Recalling the definition \eqref{support} of the support function of a zonoid $Z_\mu$ with generating measure $\mu$, we see that 
$$d_H(Z_\mu,Z_\nu)=\Big\|\frac{1}{ 2}\,\int\limits_{\SS^{n-1}}|\langle \,\cdot\,,v\rangle|\,\mu(\dint v)-\frac{1}{2}\,\int\limits_{\SS^{n-1}}|\langle \,\cdot\,,v\rangle|\,\nu(\dint v)\Big\|_{\infty}=\big\|T_\Gamma[\tau]\big\|_\infty$$ with $\Gamma(t):=|t|/2$ and $\tau:=\mu-\nu$. This special choice corresponds to the cosine transform treated in \cite{BoLi} and it is known from that paper that the multipliers $a_{n,k}(\Gamma)$ of $T_\Gamma$ satisfy $a_{n,0}(\Gamma)\neq 0$ and the estimate $|a_{n,k}(\Gamma)^{-1}|\leq c_3\,k^{(n+2)/2}$ for any even $k\in\NN$ and a universal constant $c_3>0$ (we have $a_{n,k}(\Gamma)=0$ if $k$ is odd). An application of Theorem 5.1 in \cite{HS} then shows that for each $0<\beta<2/(n+4)$ there is a constant $c_4>0$ only depending on $n$ and $\beta$ such that
$$
\Big|\int\limits_{\SS^{n-1}} f\,\dint\mu-\int\limits_{\SS^{d-1}} f\,\dint\nu\Big|=\Big|\int\limits_{\SS^{n-1}} f\,\dint\tau\Big|\le c_4\,\|f\|_{BL}\big\|T_\Gamma[\tau]\big\|_\infty^{\beta}=c_4\,\|f\|_{BL}\,d_H(Z_\mu,Z_\nu)^{\beta}
$$
for all functions $f:\SS^{n-1}\to\RR$ (the additional factor involving the total variation distance of the 
given measures $\mu$ and $\nu$  in \cite{HS} can be bounded from above by $2R$). Taking the supremum over all functions $f$ with $\|f\|_{BL}\le 1$, we obtain in particular \eqref{stab1} with $\alpha_1=1/(n+4)$ (in fact, we can  choose any number between $0$ and $2/(n+4)$).

To obtain \eqref{stab2}, we first notice that $K\subset RB^n$ implies $\|h(K,\,\cdot\,)\|_{BL}\le 2R$ by 
\cite[Lemma 1.8.12]{Schneider93}. Therefore, using  mixed volumes as introduced in Section \ref{sec:preliminaries}, we get for convex sets $K,M$ and $B^n$ that
\begin{align*}
&\big| V(K[m+1],B^n[n-m-1]) - V(K,M[m],B^n[n-m-1])\big| \cr
& \qquad = \Big| \frac{1}{n}\int\limits_{\SS^{n-1}} h(K,u) \, S_m(K,\dint u)-\frac{1}{ n}\int\limits_{\SS^{d-1}}h(K,u)\,S_m(M,\dint u)\Big| \le   c_5\, \varepsilon
\end{align*}
with a constant $c_5:=c_5(n,R)>0$, depending only on $n$ and $R$, and $\varepsilon := d_{BL}\big(S_m(K,\,\cdot\,),S_m(M,\,\cdot\,)\big)$.
Similarly, one proves the inequality
$$
\big|V(M[m+1],B^n[n-m-1]) - V(M,K[m],B^n[n-m-1])\big|  \le   c_5 \, \varepsilon \,.
$$
These are the inequalities (8.56) in \cite{Schneider93} with $m-1$ there replaced by $m$. We can now follow the proof of Lemma 8.5.2 and Theorem 8.5.4 in \cite{Schneider93} and the literature cited there (see also Theorem 7.2.6 in the first edition of \cite{Schneider93}) to get \eqref{stab2} with $\alpha_2 = 1/((n+1)2^{m-1})$ (again, $m-1$ there has to be replaced by $m$). This completes the proof of \eqref{eq:GeneralStabilityEstimate}.
\end{proof}

The bounded Lipschitz distance used in Theorem \ref{stability} can be replaced by the Prohorov distance. 

\begin{corollary}\label{cor:StabilityProhorov}
Let $m\in\{1,\dots, n-1\}$, and let $\mu,\nu\in\bM_{\rm e}(\rho,R)$ be even Borel measures on $\SS^{n-1}$. Then there is a constant $c:=c(n,\rho,R)>0$, depending only on $n,\rho$ and $R$, such that
\begin{equation}\label{eq:GeneralStabilityEstimateProhorov}
d_{P}(\mu,\nu) \le c\, d_{P}\big(S_m(Z_\mu,\,\cdot\,),S_m(Z_\nu,\,\cdot\,)\big)^{c(m,n)}\,,
\end{equation}
where $c(m,n)=\left(2^{m}(n+1)(n+4)\right)^{-1}$.
\end{corollary}
\begin{proof}
The assertion follows from Theorem \ref{stability} by applying Lemma \ref{compare} twice. For this, it remains to ensure that 
 $S_m(Z_\mu,\SS^{n-1}),S_m(Z_\nu,\SS^{n-1})$ can be bounded from above and from below in terms of $\rho,R$. Since $\mu\in\bM_{\rm e}(\rho,R)$, it follows that $\rho B^n\subset 2Z_\mu\subset RB^n$, and therefore \eqref{eq:IntVolMixedVolRelation}, \eqref{eq:VmSmRelation} and the monotonicity of mixed volumes imply that
 $$n2^{-m}\rho^m\kappa_n\le nV(Z_\mu[m],B^n[n-m])=S_m(Z_\mu,\SS^{n-1})\le n2^{-m} R^m\kappa_n\,.$$
 Lemma \ref{compare} with $\rho$ and $R$ there replaced by $n2^{-m}\rho^m\k_n$ and $n2^{-m}R^m\k_n$, respectively, provides the desired bounds.
\end{proof}



\section{Intersection}\label{sec:Intersection}

In this section, we generalize some results on intersection processes which can be found in the literature, for example in \cite{SW}. Moreover, we prove a uniqueness theorem as well as a stability result. In particular, we relax the usual stationarity condition by assuming weak stationarity only.

\subsection{Intersection processes}\label{subsec:IntersectionProcesses}

For $r\ge 2$ and $i=1,\ldots,r$, let $X_i$ be a weakly stationary process of $k_i$-dimensional flats in $\RR^n$, $1\le k_i\le n-1$, with intensity measure $\Theta_i\neq 0$, and such that $k_1+\cdots +k_r\ge (r-1)n$. Let $\gamma_i>0$ be the corresponding intensity and $\QQ_i$ the directional distribution according to \eqref{eq:Intensity}. By $Y=X_1\cap\cdots\cap X_r$ we denote the {\it intersection process} of $X_1,\ldots,X_r$. It consists of all intersections $E_1\cap\cdots\cap E_r$ of flats $E_1\in X_1,\ldots,E_r\in X_r$ which are in general position, i.e.,
$$Y(\,\cdot\,)=\sum_{E_1\in X_1,\ldots,E_r\in X_r}{\bf 1}\{E_1\cap\ldots\cap E_r\in\,\cdot\,\}\,{\bf 1}\{[E_1,\ldots,E_r]>0\}\,.$$
If there are no such intersections, then $Y$ is the empty process, otherwise $Y$ is a process of $q$-flats with $q :=k_1+\cdots +k_r-(r-1)n$, and we write $\Theta_Y$ for its intensity measure.
In the following, we do not discuss the most general result, but concentrate on two special situations, when the processes $X_1,\ldots, X_r$ are independent and when $X_1=\cdots =X_r=X$ (hence $k_1=\cdots =k_r=k$) and $X$ is of type $(S_r)$, that is,  $\Theta^{(r)}=\Theta^r$. In the latter case, we denote the intersection process of 
order $r$ by $X_{(r)}$. Formally, $X_{(r)}$ is defined as
$$X_{(r)}(\,\cdot\,)=\frac{1}{ r!}\sum_{(E_1,\ldots,E_r)\in X_{\neq}^r}{\bf 1}\{E_1\cap\ldots\cap E_r\in\,\cdot\,\}\,{\bf 1}\{[E_1,\ldots,E_r]>0\}\,.$$

In the first situation, the process $Y$ is weakly stationary. In fact, using the independence of the processes $X_1,\ldots,X_r$ together with the fact that each of the measures $\Theta_i$, $1\leq i\leq r$, is translation invariant we see that, for each $z\in\RR^n$,
\begin{align*}
&\int\limits_{A(n,q)} f(E+z)\,\Theta_Y(\dint E)\\
&= \int\limits_{A(n,k_1)}\ldots\int\limits_{A(n,k_r)} f\big((E_1\cap\ldots\cap E_r)+z\big)\,{\bf 1}\{[E_1,\ldots,E_r]>0\}\,\Theta_r(\dint E_r)\ldots\Theta_1(\dint E_1)\\
&=\int\limits_{A(n,k_1)}\ldots\int\limits_{A(n,k_r)} f\big((E_1+z)\cap\ldots\cap (E_r+z)\big)\,{\bf 1}\{[E_1,\ldots,E_r]>0\}\,\Theta_r(\dint E_r)\ldots\Theta_1(\dint E_1)\\
&=\int\limits_{A(n,k_1)}\ldots\int\limits_{A(n,k_r)} f(E_1\cap\ldots\cap E_r)\,{\bf 1}\{[E_1,\ldots,E_r]>0\}\,\Theta_r(\dint E_r)\ldots\Theta_1(\dint E_1)\\
&=\int\limits_{A(n,q)} f(E)\,\Theta_Y(\dint E)\,,
\end{align*}
for any non-negative measurable function $f:A(n,q)\to\RR$.
Thus, $Y$ is a weakly stationary $q$-flat process and we shall write $\g_Y$ for the intensity and $\QQ_Y$ for the directional distribution of $Y$, recall \eqref{eq:GammaQQ}. 

\begin{theorem}\label{thm:IntersectionK1Kr} Let $k_1,\ldots ,k_r\in\{1,\ldots ,n-1\}$ be such that $k_1+\ldots+k_r\geq (r-1)n$ and put $q:=k_1+\cdots +k_r-(r-1)n$. For $i=1,\ldots,r$, let $X_i$ be a weakly stationary $k_i$-flat process with intensity $\gamma_i$ and directional distribution $\QQ_i$, and assume that $X_1,\ldots,X_r$ are independent.
Then, for all Borel sets $A\subset G(n,q)$,
$$\g_{Y}\,\QQ_{Y}(A)={\g_1\cdots\g_r}\,\int\limits_{G(n,k_1)}\cdots\int\limits_{G(n,k_r)}{\bf 1}\left\{L_1\cap\ldots\cap L_r\in A\right\}\,[L_1,\ldots,L_r]\,\QQ_r(\dint L_r)\ldots\QQ_1(\dint L_1)\,.$$
\end{theorem}
Notice that the formula in Theorem \ref{thm:IntersectionK1Kr} is still valid if one of the $\g_i$'s is zero. In this case, the left- and the right-hand side are both equal to zero. 
\begin{proof}[Proof of Theorem \ref{thm:IntersectionK1Kr}]
Using the definition \eqref{eq:Intensity} of the measures $\Theta_1,\ldots,\Theta_r$, relations \eqref{eq:GammaQQ} and the independence of $X_1,\ldots,X_r$,
we obtain that $\k_{n-q}\g_{Y}\,\QQ_{Y}(A) $ equals
\begin{equation*}
\begin{split}
&\EE\sum_{E_1\in X_1,\ldots,E_r\in X_r}{\bf 1}\left\{E_1\cap\ldots\cap E_r\cap B^n\neq\emptyset,\,L(E_1)\cap\ldots\cap L(E_r)\in A\,,[E_1,\ldots,E_r]>0\right\}\\
&=\int\limits_{A(n,k_1)}\cdots\int\limits_{A(n,k_r)}{\bf 1}\left\{E_1\cap\ldots\cap E_r\cap B^n\neq\emptyset,\,L(E_1)\cap\ldots\cap L(E_r)\in A\right\}\\
&\hspace{7cm}\times{\bf 1}\{[E_1,\ldots,E_r]>0\}\,\Theta_r(\dint E_r)\ldots\Theta_1(\dint E_1)\\
&=\g_1\cdots\g_r\int\limits_{G(n,k_1)}\cdots\int\limits_{G(n,k_r)}\int\limits_{L_1^\perp}\cdots\int\limits_{L_r^\perp}{\bf 1}\left\{(L_1+x_1)\cap\ldots\cap(L_r+x_r)\cap B^n\neq\emptyset,\,L_1\cap\ldots\cap L_r\in A\right\}\\
&\hspace{3.5cm}\times{\bf 1}\{[L_1,\ldots,L_r]>0\}\,\ell_{L_r^\perp}(\dint x_r)\ldots\ell_{L_1^\perp}(\dint x_1)\,\QQ_r(\dint L_r)\ldots\QQ_1(\dint L_1)\,.
\end{split}
\end{equation*}
The result now follows from the identity
\begin{align*}
&\int\limits_{L_1^\perp}\cdots\int\limits_{L_r^\perp}{\bf 1}\left\{(L_1+x_1)\cap\ldots\cap(L_r+x_r)\cap B^n\neq\emptyset\right\}\,\ell_{L_r^\perp}(\dint x_r)\ldots\,\ell_{L_1^\perp}(\dint x_1)\\
&\qquad= \k_{rn-k_1-\ldots -k_r}\,[L_1,\ldots,L_r]\,,
\end{align*}
which holds for subspaces $L_1,\ldots,L_r$ satisfying $[L_1,\ldots,L_r]>0$ and which can be proved directly, but is also obtained from an iteration of the translative Crofton formula for mixed measures and a reduction formula for balls (Theorem 5.1 and Formula (7.1) in \cite{W01}).
\end{proof}

Coming now to the second situation, we consider a weakly stationary $k$-flat process $X$ in $\RR^n$ with $n/2\le k\le n-1$ and $r\ge 2$ such that $rk\geq (r-1)n$. In the following, we assume that $X$ is of type $(S_r)$.  The $r$th intersection process $X_{(r)}$ of $X$ is defined as the process of $(rk-(r-1)n)$-dimensional flats in $\RR^n$ which arise as intersections of flats $E_1,\ldots,E_r\in X$ in general position. Similarly as above, one shows that $X_{(r)}$ is weakly stationary. Let $\g_{(r)}$ be the intensity of $X_{(r)}$ and $\QQ_{(r)}$ its directional distribution. The next result is a joint generalization of Theorems 4.4.8 and 4.4.9 in \cite{SW}.

\begin{theorem}\label{thm:IntersectionDensity}
Let $n/2\le k\le n-1$ and $rk\geq (r-1)n$. Let $X$ be a weakly stationary $k$-flat process of type $(S_r)$ with intensity $\g$ and directional distribution $\QQ$. Then, for all Borel sets $A\subset G(n,rk-(r-1)n)$, 
\begin{equation}\label{eq:FormulaIntersecionDensity}
\g_{(r)}\,\QQ_{(r)}(A)=\frac{\g^r}{ r!}\int\limits_{G(n,k)}\cdots\int\limits_{G(n,k)}{\bf 1}\left\{L_1\cap\ldots\cap L_r\in A\right\}\,[L_1,\ldots,L_r]\,\QQ(\dint L_r)\ldots\QQ(\dint L_1)\,.
\end{equation}
\end{theorem}
\begin{proof} Again by \eqref{eq:Intensity}, \eqref{eq:GammaQQ}, the definition of $X_{(r)}$ and \eqref{eq:Campbell}, and since $n-(rk-(r-1)n)=r(n-k)$, we obtain that $r!\k_{r(n-k)}\,\g_{(r)}\,\QQ_{(r)}(A)$ equals
\begin{equation*}
\begin{split}
&\EE\sum_{(E_1,\ldots,E_r)\in X^r_{\not=}}{\bf 1}\left\{E_1\cap\ldots\cap E_r\cap B^n\neq\emptyset,\,L(E_1)\cap\ldots\cap L(E_r)\in A\,,[E_1,\ldots,E_r]>0\right\}\\
&=\int\limits_{A(n,k)}\cdots\int\limits_{A(n,k)}{\bf 1}\left\{E_1\cap\ldots\cap E_r\cap B^n\neq\emptyset,\,L(E_1)\cap\ldots\cap L(E_r)\in A\right\}\\
&\hspace{7cm}\times{\bf 1}\{[E_1,\ldots,E_r]>0\}\,\Theta(\dint E_r)\ldots\Theta(\dint E_1)\\
& =\g^r\int\limits_{G(n,k)}\cdots\int\limits_{G(n,k)}\int\limits_{L_1^\perp}\cdots\int\limits_{L_r^\perp}{\bf 1}\left\{(L_1+x_1)\cap\ldots\cap(L_r+x_r)\cap B^n\neq\emptyset,\,L_1\cap\ldots\cap L_r\in A\right\}\\ &\hspace{3.5cm}\times{\bf 1}\{[L_1,\ldots,L_r]>0\}\,\ell_{L_r^\perp}(\dint x_r)\ldots\ell_{L_1^\perp}(\dint x_1)\,\QQ(\dint L_r)\ldots\QQ(\dint L_1)\,.
\end{split}
\end{equation*}
Now the result follows as in the proof of Theorem \ref{thm:IntersectionK1Kr}.
\end{proof}

\subsection{Uniqueness problems and stability estimates}\label{sec:UniquenessHyperplanes}

In this section, we consider the problem whether, for a weakly stationary $k$-flat process $X$ of type $(S_r)$, the directional distribution $\QQ$ of $X$ is determined by the directional distribution $\QQ_{(r)}$ of the $r$-th intersection process $X_{(r)}$ or whether the intensity measure $\Theta_{(r)}$ of $X_{(r)}$ determines the intensity measure $\Theta$ of $X$. 

Let us begin with the case $k=1$, which leads to a non-trivial intersection process if and only if $n=r=2$. However, the intersection process $X_{(2)}$ then is an ordinary point process in the plane and its intensity $\g_{(2)}$ does not allow to determine the directional distribution $\QQ$ of $X$. For example, any rotation of $\QQ$ leads to a process $X'$ with the same intersection intensity. More generally, let $n$ be even and $k=n/2$. Let $U_1,U_2\in G(n,k)$ be in general position, that is, $U_1\cap U_2=\{0\}$. Further, let $\alpha\in (0,1)$ and 
$\QQ:=\alpha\delta_{U_1}+(1-\alpha)\delta_{U_2}$. If $X$ is a Poisson process of $k$-flats with intensity $\gamma$ and direction 
distribution $\QQ$, then $\gamma_{(2)}=\gamma^2\alpha(1-\alpha)[U_1,U_2]$. Hence there is a continuum of choices of these parameters leading to the same intersection density $\gamma_{(2)}$.

We next consider an intermediate regime for the dimension parameter $k$. It covers for example the case $n=5$ and $k=3$, but not $n=6$ and $k=4$. 
The following theorem shows that a uniqueness result does not hold in a certain range of values of $k$, even for Poisson processes, but the exact conditions on $k$ under which a non-uniqueness result holds remain open. Recall that a weakly stationary Poisson process is automatically stationary.

\begin{theorem}\label{thmuniquedirect1}
Let $k\in\{2,\ldots,n-2\}$ be such that $\frac{n}{2}\leq k< \frac{n}{2}+\left\lfloor \frac{n-k}{2}\right\rfloor $. Then there exist two stationary Poisson processes $X_1$ and $X_2$ of $k$-flats with the same intensity $\g>0$, but different directional distributions $\QQ_1$ and $\QQ_2$, which are absolutely continuous with respect to the Haar probability measure on $G(n,k)$ and for which the associated second intersection processes have the same intensity measure.
\end{theorem}
\begin{proof}
In view of Theorem \ref{thm:IntersectionDensity} the assertion follows once we have shown that there are two different probability measures $\QQ_1$ and $\QQ_2$ on $G(n,k)$ which are absolutely continuous with respect to the Haar probability measure on $G(n,k)$ and such that
\begin{equation}\label{nonunique0}
\int\limits_{G(n,k)}\int\limits_{G(n,k)}  g(L\cap M)\,[L,M]\,\QQ_1(\dint L)\, \QQ_1(\dint M)
=\int\limits_{G(n,k)}\int\limits_{G(n,k)} g(L\cap M)\,[L,M]\, \QQ_2(\dint L)\, \QQ_2(\dint M)
\end{equation}
for all continuous functions $g:G(n,2k-n)\to\RR$. In order to prove the existence of two such measures, we apply the main 
result in \cite{GHR96}, as well as the equivalence preceding it, with $j=2k-n$. We notice that the conditions given there for $j$ are satisfied. Namely, since $k\ge n/2$, we have $j\ge 0$. Also, since $k< \left\lfloor \frac{n-k}{2}\right\rfloor +\frac{n}{2}$, we have $j<2\left\lfloor \frac{r}{2}\right\rfloor$, since the rank $r$ in the sense of \cite{GHR96} of $G(n,k)$ is $n-k$ (here we use $k\ge n/2$ again).  Thus, we obtain the existence of a non-zero finite signed measure $\QQ$ on $G(n,k)$, which is absolutely continuous with respect to the Haar probability measure on $G(n,k)$, such that 
$$
\int\limits_{G(n,k)}g(L\cap M)\,[L,M]\, \QQ(\dint L)=0
$$
for all continuous functions $g$ on $G(n,2k-n)$ and all $M\in G(n,k)$. 
Since $0\neq \QQ=\QQ_1-\QQ_2$ with two non-negative finite measures $\QQ_1\neq \QQ_2$ on $G(n,k)$, which are absolutely continuous with respect to the Haar probability measure on $G(n,k)$, we deduce that
\begin{equation}\label{eqeq}
\int\limits_{G(n,k)}g(L\cap M)\,[L,M]\, \QQ_1(\dint L)=
\int\limits_{G(n,k)}g(L\cap M)\,[L,M]\, \QQ_2(\dint L)
\end{equation}
for all continuous functions $g$ on $G(n,2k-n)$ and all $M\in G(n,k)$. Here we may and will assume that $\QQ_1$ is a probability  measure. Choosing $g=1$ and integrating over $M$ with the Haar measure $\nu_k$, we see that $\QQ_2$ is then a probability measure, too. 
Applying \eqref{eqeq} twice and using Fubini's theorem, we get
\begin{align*}
\int\limits_{G(n,k)}\int\limits_{G(n,k)} &g(L\cap M)\,[L,M]\, \QQ_1(\dint L)\, \QQ_1(\dint M)\\ &=\int\limits_{G(n,k)}\int\limits_{G(n,k)}g(L\cap M)\,[L,M]\, \QQ_2(\dint L)\, \QQ_1(\dint M)\\
&=\int\limits_{G(n,k)}\int\limits_{G(n,k)}g(L\cap M)\,[L,M]\, \QQ_1(\dint M)\, \QQ_2(\dint L)\\
&=\int\limits_{G(n,k)}\int\limits_{G(n,k)}g(L\cap M)\,[L,M]\, \QQ_2(\dint M)\, \QQ_2(\dint L)\\
&=\int\limits_{G(n,k)}\int\limits_{G(n,k)}g(L\cap M)\,[L,M]\, \QQ_2(\dint L)\, \QQ_2(\dint M),
\end{align*}
for all continuous functions $g$ on $G(n,2k-n)$, as required.
\end{proof}

Let us now turn to hyperplane processes, that is, $k$-flat processes with $k=n-1$. In this case we obtain a uniqueness result which has not been stated in the literature so far, dealing with higher-order intersection processes. Thus, let $X$ be a weakly stationary hyperplane process in $\RR^n$ with intensity $\g>0$ and associated even spherical directional distribution $\QQ$ (recall the discussion related to \eqref{symhat}).

The intensity measure $\Theta$ of $X$ is then given by
\begin{equation}\label{eq:IntensityMeasureHyperplanes}
\int\limits_{A(n,n-1)} f(H)\,\Theta(\dint H)=\gamma\int\limits_{\SS^{n-1}}\int\limits_{\RR} f(u^\perp+tu)\,\dint t\,\QQ(\dint u)\,,
\end{equation}
where $f$ is a non-negative measurable function on $A(n,n-1)$. 
We shall assume that $\QQ$ is {\it non-degenerate} in the sense that it is  not concentrated on a great-subsphere of $\SS^{n-1}$.  For fixed  $r\in\{2,\ldots,n\}$,  we consider the $r$-th intersection process $X_{(r)}$ of $X$, which is a weakly stationary process of $(n-r)$-flats in $\RR^n$. We also assume that the hyperplane process $X$ is of type $(S_r)$, hence its $r$-th factorial moment measure  $\Theta^{(r)}$ satisfies $\Theta^{(r)}=\Theta^r$. In this case, Theorem \ref{thm:IntersectionDensity} shows that the intensity $\gamma_{(r)}$ and the directional distribution $\QQ_{(r)}$ of $X_{(r)}$ can be expressed in the form
\begin{equation}\label{eq:QkHyperplane}
\gamma_{(r)}\,\QQ_{(r)}(A)=\frac{\gamma^r}{r!}\int\limits_{\SS^{n-1}}\ldots \int\limits_{\SS^{n-1}}\mathbf{1}\big\{
u_1^\perp\cap\ldots\cap u_r^\perp\in A\big\}\,\nabla_r(u_1,\ldots,u_r)\, \QQ(\dint u_r)\ldots \QQ(\dint u_1)\,,
\end{equation}
where $A\subset G(n,n-r)$ is a Borel set and  $\nabla_r$ has been defined in Section \ref{sec:preliminaries}. We now show that $\QQ_{(r)}$ uniquely determines $\QQ$. 

\begin{theorem}\label{thmuniquedirect2}
Let $r\in\{2,\ldots,n-1\}$, and let $X$ be a weakly stationary hyperplane process in $\RR^n$ of type $(S_r)$ with intensity $\gamma>0$ and non-degenerate spherical directional distribution $\QQ$. Then the intensity measure $\Theta_{(r)}$ of $X_{(r)}$ uniquely determines the intensity measure $\Theta$ of $X$. In particular, $\Theta_{(r)}$ is rotation invariant if and only if $\Theta$ is rotation invariant. Moreover, the directional distribution $\QQ_{(r)}$ of the $r$-th intersection process $X_{(r)}$ uniquely determines $\QQ$.  
\end{theorem}

\begin{proof}
For a Borel set $A\subset G(n,n-r)$ we define 
\begin{equation}\label{eq:defMuQ}
\mu_\QQ^{(r)}(A):=\int\limits_{\SS^{n-1}}\ldots \int\limits_{\SS^{n-1}}\mathbf{1}\{u_1^\perp\cap\ldots\cap u_r^\perp\in A\}\,\nabla_r(u_1,\ldots,u_r)\, \QQ(\dint u_r)\ldots \QQ(\dint u_1)\,.
\end{equation}
In addition, we consider the transformation $T_{n-r}:\bM(G(n,n-r))\to\bM_{\rm e}(\SS^{n-1})$ given by
$$
T_{n-r}(\mu)(C):=\int\limits_{G(n,n-r)}\int\limits_{\SS_U}\mathbf{1}\{v\in C\}\, \sigma_U(\dint v)\, \mu (\dint U)\,,
$$
where $\mu\in \bM(G(n,n-r))$ and $C\subset\SS^{n-1}$ is a Borel set. Let $Z_\mathbb{Q}$ be the zonoid associated with $\mathbb{Q}$. Then, by a special case of 
\cite[Theorem 8]{HTW}, we get 
\begin{equation}\label{dh1}
T_{n-r}(\mu_\QQ^{(r)})=r!\binom{n-1}{r}\,S_r(Z_\QQ,\,\cdot\,)\,.
\end{equation}
This result can also be deduced from formula (4.62) in \cite{SW} (and the statement thereafter). Namely, this formula shows that $\mu_\QQ^{(r)}$ is (up to orthogonality) proportional to the $r$-th projection generating measure of $Z_\QQ$, and therefore 
it follows from the results in \cite{Weil1979} that $T_{n-r}(\mu_\QQ^{(r)})$ is proportional to $S_r(Z_\QQ,\,\cdot\,)$. For yet another approach, see \cite[Theorem 5.3.4]{Schneider93}.

Hence we get
\begin{equation}\label{determine1}
T_{n-r}\left(\gamma_{(r)}\QQ_{(r)}\right)=T_{n-r}\left(\frac{\gamma^r}{r!}\mu_\QQ^{(r)}\right)
=\binom{n-1}{r}\gamma^r\,S_r(Z_\QQ,\,\cdot\,)=\binom{n-1}{r}\,S_r(Z_{\gamma\QQ},\,\cdot\,)\,,
\end{equation}
where we used the linearity of $T_{n-r}$ and the homogeneity of degree $r$ of the $r$-th area measure. 

Assume now that the intensity measure $\Theta_{(r)}$ of $X_{(r)}$ is given. Since  $\Theta_{(r)}$ determines $\gamma_{(r)}\QQ_{(r)}$, 
\eqref{determine1} shows that then also the area measure $S_r(Z_{\gamma\QQ},\,\cdot\,)$ is determined. An application of 
\cite[Corollary 8.1.4]{Schneider93} yields that the centrally symmetric convex body $Z_{\gamma \QQ}$ is determined. Note that $Z_{\gamma \QQ}$ is full-dimensional, since $\QQ$ is non-degenerate and hence $\g_{(r)}\neq 0$. From \cite[Theorem 3.5.4]{Schneider93} we conclude that then 
$\gamma\QQ$ is determined. Since $\QQ$ is a probability measure,  we finally get that $\gamma$ and $\QQ$ are uniquely determined.

Now we assume that $\QQ_{(r)}$ is given. Then \eqref{determine1} shows that
$$
\frac{S_r(Z_\QQ,\,\cdot\,)}{S_r(Z_\QQ,\SS^{n-1})}=\frac{T_{n-r}\left(\QQ_{(r)}\right)}{\omega_{n-r}}\,,
$$
where we used that $T_{n-r}(\QQ_{(r)})(G(n,n-r))=\omega_{n-r}$. Hence, $\QQ_{(r)}$ determines $c^r\, S_r(Z_\QQ,\,\cdot\,)$ with 
$c:=S_r(Z_\QQ,\SS^{n-1})^{-1/r}$, and therefore also $S_r(Z_{c\,\QQ},\, \cdot\,)$. By the same arguments as before, this implies that 
$c\,\QQ$ is determined. Since $\QQ$ is a probability measure, we conclude that $\QQ$ is determined as well.

The remaining assertions follow immediately.
\end{proof}

We have the following consequence if $X$ is a Poisson process.

\begin{corollary}
Let $X$ be a stationary Poisson hyperplane process in $\RR^n$ with intensity $\g>0$ and non-degenerate directional distribution. Then, for any $r\in\{2,\ldots,n-1\}$, the distribution of $X$ is uniquely determined by the knowledge of $\g_{(r)}$ and $\QQ_{(r)}$. Moreover, $X$ is isotropic if and only if $X_{(r)}$ is isotropic.
\end{corollary}

Let us refine and strengthen the preceding uniqueness result in the form of a stability estimate. Here the situation can be described as follows. Let 
$X$ and $X'$ be two 
weakly stationary hyperplane processes of type $(S_r)$ in $\RR^n$ with (positive) intensities $\g$ and $\g'$ and direction distributions $\QQ$ and $\QQ'$, respectively. Let $\g_{(r)}$ denote the intensity and $\QQ_{(r)}$ the direction distribution of the $r$-th intersection process
$X_{(r)}$ of $X$, and let $\g_{(r)}'$ and $\QQ_{(r)}'$ be the corresponding quantities for $X_{(r)}'$. Suppose now that $\g_{(r)}\QQ_{(r)}$ 
and  $\g_{(r)}'\QQ_{(r)}'$ are close in some (quantitative) sense. Does this imply that also $\g\QQ$ and $\g'\QQ'$ are close to each other?

In view of Theorem \ref{thmuniquedirect2} and Corollary \ref{cor:StabilityProhorov} and the connection between directional distributions and area measures, we can get an answer if the closeness is measured by means of the bounded Lipschitz or the Prohorov distance. For that purpose, we have to specify a metric on $G(n,k)$ with $k\in \{1,\ldots,n-1\}$. We choose the metric $\Delta$ induced by the notion of a \textit{direct rotation} (see \cite{HS2010} and the literature cited therein). Let $O_n$ denote the rotation group of $\RR^n$. For a rotation 
$\rho\in O_n$, let $M_\rho=((M_\rho)_{ij})_{i,j=1}^n$ be the matrix of $\rho$ with respect to an orthonormal basis $e_1,\ldots,e_n$ of $\RR^n$, that is, $(M_\rho)_{ij}=\langle \rho(e_i),e_j\rangle$, and let $I=(\delta_{ij})_{i,j=1}^n$ denote the identity matrix. Then  
$$
|\rho|:=\sum_{i,j=1}^n((M_\rho)_{ij}-\delta_{ij})^2
$$
is independent of the chosen orthonormal basis. 
For any  $u\in\mathbb{S}^{n-1}$, we have $\|\rho u-u\|\le|\rho|$. A metric on $G(n,k)$ is then defined by   
$$
\Delta(U_1,U_2):=\min\{|\rho|:\rho\in O_n,\,U_2=\rho U_1\},\qquad U_1,U_2\in G(n,k)\,.
$$
The topology induced by this metric is the standard topology on $G(n,k)$, see \cite{HS2010} for further details. A comparison 
of different notions of distances between subspaces is provided in the survey \cite[Section 4.3]{Edelman} and in \cite{HammLee}.

\begin{theorem}\label{stability2} Let $r\in\{2,\dots ,n-1\}$, and let $X, X'$ be two weakly stationary hyperplane processes in $\RR^n$ of type $(S_r)$ with intensities $\gamma,\g'>0$ and spherical directional distributions $\QQ,\QQ'$, respectively. Assume that 
$\g\QQ, \g'\QQ'\in\bM_{\rm e}(\rho,R)$ with some constants $0<\rho\le R<\infty$. Then there is a constant $c>0$, depending only on $n,\rho$ and $R$, such that
\begin{equation}\label{eqstab1}
d_{BL}(\g\QQ,\g'\QQ') \le c \,  d_{BL}(\g_{(r)}\QQ_{(r)},\g_{(r)}'\QQ_{(r)}')^{2c(r,n)}
\end{equation}
and
\begin{equation}\label{eqstab2}
d_P(\g\QQ,\g'\QQ') \le c \, d_P(\g_{(r)}\QQ_{(r)},\g_{(r)}'\QQ_{(r)}')^{c(r,n)}\,,
\end{equation}
where $c(r,n)=(2^{r}(n+1)(n+4))^{-1}$.
\end{theorem}

\begin{proof} We use the same notation as in the proof of Theorem \ref{thmuniquedirect2}. From \eqref{determine1} 
we have
\begin{equation}\label{eqneu}
T_{n-r}(\g_{(r)}\QQ_{(r)}) = \frac{1}{ r!}\,T_{n-r}(\mu_{\g\QQ}) = \binom{n-1}{ r}\,S_r(Z_{\g\QQ},\,\cdot\,).
\end{equation}
Since $\g\QQ, \g'\QQ'\in\bM_{\rm e}(\rho,R)$, we can first apply Theorem \ref{stability}  and then \eqref{eqneu} to get
\begin{align*}
d_{BL}(\g\QQ,\g'\QQ') &\le c_1\, d_{BL}\big(S_r(Z_{\g\QQ},\,\cdot\,),S_r(Z_{\g'\QQ'},\,\cdot\,)\big)^{2c(r,n)}\\
&\le c_2\, d_{BL}\big(T_{n-r}(\g_{(r)}\QQ_{(r)}),T_{n-r}(\g_{(r)}'\QQ_{(r)}')\big)^{2c(r,n)}
\end{align*}
with constants $c_1,c_2>0$, depending only on $n,\rho$ and $R$. 
Consequently, the assertion of part (a) of the theorem follows once we have shown that 
\begin{equation}\label{T-stability}
d_{BL}(T_{n-r}\mu,T_{n-r}\nu) \le \omega_{n-r}\, d_{BL}(\mu,\nu)
\end{equation}
holds for arbitrary finite Borel measures $\mu,\nu$ on $G(n,n-r)$. For this, let $f:\SS^{n-1}\to\RR$ be a function with $\|f\|_{BL}\le 1$ and define $F_f:G(n,n-r)\to\RR$ by 
$$
F_f(U):=\int\limits_{\SS_U}f(v)\, \sigma_U(\dint v),\qquad U\in G(n,n-r).
$$
Clearly, $\|F_f\|_\infty\le \omega_{n-r}\|f\|_\infty$. Let $U_1,U_2\in G(n,n-r)$ with $\Delta(U_1,U_2)=|\rho|$ for some rotation $\rho\in O_n$ satisfying $U_2=\rho U_1$. Since $|f(v)-f(\rho v)|\le \|f\|_L\|v-\rho v\|\le \|f\|_L|\rho|$ for all $v\in \mathbb{S}^{n-1}$, we get
\begin{align*}
|F_f(U_1)-F_f(U_2)|&=\Big|\int\limits_{\SS_{U_1}}f(v)\, \sigma_{U_1}(\dint v)-\int\limits_{\SS_{U_2}}f(v)\, \sigma_{U_2}(\dint v)\Big|\\
&=\Big|\int\limits_{\SS_{U_1}}f(v)\, \sigma_{U_1}(\dint v)-\int\limits_{\SS_{U_1}}f(\rho v)\, \sigma_{U_1}(\dint v)\Big|\\
&\le \int\limits_{\SS_{U_1}}\big| f(v)-f(\rho v)\big|\,\sigma_{U_1}(\dint v)\\
&\leq \o_{n-r}\,|\rho|\,\|f\|_L\\
&=\omega_{n-r}\,\Delta(U_1,U_2)\,\|f\|_L\,.
\end{align*}
This shows that $\|F_f\|_L\le \omega_{n-r}\,\|f\|_L$, and thus also $\|F_f\|_{BL}\le \omega_{n-r}\,\|f\|_{BL}$. Therefore 
\begin{align*}
\Big|\int\limits_{\SS^{n-1}}f(v)\,(T_{n-r}\mu)(\dint v)-\int\limits_{\SS^{n-1}}f(v)\,(T_{n-r}\nu)(\dint v)\Big|&=\Big|\int\limits_{G(n,n-r)}F_f(U)\,\mu(\dint U)-\int\limits_{G(n,n-r)}F_f(U)\,\nu(\dint U)\Big|\\
&\le\|F_f\|_{BL}\, d_{BL}(\mu,\nu)\\
&\le \omega_{n-r}\,\|f\|_{BL}\,d_{BL}(\mu,\nu)\,,
\end{align*}
which proves \eqref{T-stability}. Thus part (a) of the theorem is established.

Part (b) follows from Lemma \ref{compare} once we have shown that in addition to $\g\QQ,\g'\QQ'\in\bM_{\rm e}(\rho,R)$ we 
also have $\g_{(r)}\QQ_{(r)},\g_{(r)}'\QQ_{(r)}'\in\bM(\bar\rho,\bar R)$ on $G(n,n-r)$ with some constants $0<\bar\rho\le\bar R<\infty$,  
where $\bar\rho$ depends only on $\rho$ and $\bar R$ depends only on $R$. To see that this is indeed the case, we use \eqref{determine1}, \eqref{eq:IntVolMixedVolRelation} and \eqref{eq:VmSmRelation} to find that
\begin{align*}
\left(\g_{(r)}\QQ_{(r)}\right)(G(n,n-r))
&=\frac{1}{\omega_{n-r}}\binom{n-1}{r} S_r(Z_{\gamma \QQ},\SS^{n-1})\\
&=\frac{1}{\omega_{n-r}}\binom{n-1}{r}n V(Z_{\gamma \QQ}[r],B^n[n-r])\,.
\end{align*}
Since $\g\QQ \in\bM_{\rm e}(\rho,R)$, we can now argue as in the proof of Corollary 
\ref{cor:StabilityProhorov}.
\end{proof}


\section{Proximity}\label{sec:proximity}

In this section we introduce proximity processes. In the stationary case, the intensity of such processes has first been studied in \cite{Schneider99}, see also \cite{SW}. We present generalizations in different directions and also investigate uniqueness and stability problems.

\subsection{The proximity process}\label{sec:intensitymeasure}

Let $X_1$ be a weakly stationary $k_1$-flat process and $X_2$ a weakly stationary $k_2$-flat process in $\RR^n$ such that $k_1+k_2<n$ and $k_1,k_2\ge 1$. Note that this is dual to the situation considered in Section \ref{sec:Intersection}, where $k_1$ and $k_2$ were chosen such that $k_1+k_2\geq n$. The intensity measures of $X_1$ and $X_2$ are denoted by $\Theta_1$ and $\Theta_2$, respectively. In the following, we assume that $\Theta_1,\Theta_2\neq 0$. In this case $\Theta_1$ and $\Theta_2$ admit a factorization as in \eqref{eq:Intensity} with intensities $\g_1,\g_2>0$ and directional distributions $\QQ_1,\QQ_2$, respectively.

The \textit{proximity process} $\Phi$ associated with $X_1$ and $X_2$ is the random collection of line segments $s=\overline{x_Ex_F}$, the perpendicular of $E$ and $F$, for which $E\in X_1$ and $F\in X_2$ are disjoint, in general position and such that $d(E,F)=\|x_E-x_F\|\leq\delta$ for some prescribed distance threshold $\delta>0$. Throughout the following, we keep $\delta$ fixed. By $\L$ we denote the intensity measure of $\Phi$.  As indicated in Section \ref{sec:preliminaries}, $\Phi$ and $\Lambda$ are considered as measures on the product space $\RR^d\times(0,\delta]\times\SS^{n-1}$, which are symmetric in the third component. More formally, we take $\Lambda$ as the intensity measure of the random measure
$$
\sum_{s\in\Phi}{\bf 1}\{(m(s),d(s))\in A\times B\}\,\frac{1}{2}\mathcal{H}^0(\phi(s)\cap C),
$$
where $A\subset\RR^n$, $B\subset(0,\infty)$ and $C\subset\SS^{n-1}$ are Borel sets. Recall that $\phi(s)$ denotes the linear subspace which is  parallel to the segment $s$. In case that $s=\overline{x_Ex_F}$, as above, we will often write $\phi(E,F)$ instead of $\phi(s)$.

\begin{remark}\rm 
The restriction $d(E,F)\leq\d$ in the definition of $\Phi$ is geometrically motivated and in fact necessary in order to interpret $\Phi$ as a segment process in $\RR^n$ as introduced in Section \ref{sec:preliminaries}. Without the distance threshold $\delta$, the union of all segments of the proximity process may be a dense subset of $\RR^n$. 
\end{remark}

Our first goal is to derive an expression for $\L$ in terms of the parameters of the underlying processes $X_1$ and $X_2$, that is to say, in terms of $\g_1,\g_2$, $\QQ_1$ and $\QQ_2$. In particular, the next result ensures that, for independent $X_1,X_2$,  the proximity process $\Phi$ is weakly stationary in the sense that the intensity measure $\Lambda$ of $\Phi$ is translation invariant (with respect to the first component). The intensity of $\Phi$ is called the \textit{proximity} associated with the two flat processes $X_1$ and $X_2$. 

\begin{theorem}\label{thm:IntensityMeasureSegmentProcess}
Let $k_1,k_2\ge 1$ with $k_1+k_2<n$, let $X_1$ be a weakly stationary $k_1$-flat process, $X_2$ a weakly stationary $k_2$-flat process in $\RR^n$ and suppose that $X_1$ and $X_2$ are independent. Then
\begin{equation*}
\begin{split}
\L(A\times B\times C)= &\,{\g_1\g_2}\,\ell(A)\int\limits_B t^{n-k_1-k_2-1}\,\dint t\\ &\qquad\times\int\limits_{G(n,k_1)}\int\limits_{G(n,k_2)}\sigma_{(L+M)^\perp}(C\cap(L+M)^\perp)\,[L,M]\,\QQ_2(\dint M)\,\QQ_1(\dint L)
\end{split}
\end{equation*}
for Borel sets $A\subset\RR^n$, $B\subset(0,\delta]$ and $C\subset\SS^{n-1}$.
If $\QQ_1$ and $\QQ_2$ are rotation invariant, then $\L$ is invariant under rotations of its third argument. 
\end{theorem}

\begin{proof}
By construction of the proximity process, one has that 
$$
2\,\L(A\times B\times C)=\,\EE\sum_{E\in X_1,F\in X_2} {\bf 1}\left\{m(E,F)\in A,\,d(E,F)\in B,\,[E,F]>0\right\} \mathcal{H}^0(\phi(E,F)\cap  C)\,.
$$ 
Note that if $E\cap F\neq\emptyset$, then the intersection is a single point and 
$d(E,F)=0$, hence the indicator function is zero, since $B\subset(0,\delta]$. 

Using the independence of $X_1$ and $X_2$ and the decomposition \eqref{eq:Intensity} of the intensity measures $\Theta_1$ and $\Theta_2$, this can be re-written as
\begin{align*}
2\,\L(A\times B\times C)&= {\g_1\g_2}\int\limits_{G(n,k_1)}\int\limits_{G(n,k_2)}\int\limits_{L^\perp}\int\limits_{M^\perp}{\bf 1}\left\{ m(L+x,M+y)\in A,\,d(L+x,M+y)\in B\right\}\\
&\qquad\times \mathcal{H}^0\left(\phi(L+x,M+y)\cap C\right)\,{\bf 1}\{[L,M]>0\}\,\ell_{M^\perp}(\dint y)\,\ell_{L^\perp}(\dint x)\,\QQ_2(\dint M)\,\QQ_1(\dint L)\, .
\end{align*}
We introduce the two subspaces $V:=L+M$ and $U:=V^\perp$ and decompose $x\in L^\perp$ and $y\in M^\perp$ according to
\begin{eqnarray*}
 x &=& x_1+x_2,\quad x_1\in L^\perp\cap V,\,x_2\in U\,,\\
 y &=& y_1+y_2,\quad y_1\in M^\perp\cap V,\,y_2\in U\,.
\end{eqnarray*}
Then 
$$
m(L+x,M+y)=z+\frac{x_2+y_2}{2},\quad d(L+x,M+y)=\|x_2-y_2\|,\quad \phi(L+x,M+y)=\phi(\overline{x_2y_2}),
$$
if $x_2\neq y_2$, 
where $z$ is the unique point satisfying  $\{z\}=(L+x_1)\cap(M+y_1)$. Thus, we find that
\begin{align*}
&2\L(A\times B\times C)= {\g_1\g_2} \int\limits_{G(n,k_1)}\int\limits_{G(n,k_2)}\int\limits_U\int\limits_U\int\limits_{L^\perp\cap V}\int\limits_{M^\perp\cap V}{\bf 1}\left\{z+\frac{x_2+y_2}{ 2}\in A,\,\|x_2-y_2\|\in B\right\}\\
 &\qquad\times \mathcal{H}^0\left(\phi(\overline{x_2y_2})\cap  C\right)\,{\bf 1}\{[L,M]>0\}\,\ell_{M^\perp\cap V}(\dint y_1)\,\ell_{L^\perp\cap V}(\dint x_1)\,\ell_U(\dint y_2)\,\ell_U(\dint x_2)\,\QQ_2(\dint M)\,\QQ_1(\dint L)\,.
\end{align*}
The inner double integral can be evaluated similarly as in \cite{Schneider99} (see also the proof of Theorem 4.4.10 in \cite{SW}). In particular, 
we use the inverse of the transformation $(L^\perp\cap V)\times (M^\perp \cap V)\to V$, $(x_1,y_1)\mapsto z$, which has Jacobian 
$[L^\perp\cap V,M^\perp\cap V]=[L,M]$. Hence, we obtain
\begin{align*}
2\,\L(A\times B\times C)&= {\g_1\g_2}\int\limits_{G(n,k_1)}\int\limits_{G(n,k_2)}\int\limits_U\int\limits_U [L,M]\, \ell_{k_1+k_2}\left(A\cap\left(V+\frac{x_2+y_2}{ 2}\right)\right){\bf 1}\{\|x_2-y_2\|\in B\}\\
&\qquad\qquad\times \mathcal{H}^0\left(\phi(\overline{x_2y_2})\cap C\right)\,\ell_U(\dint y_2)\,\ell_U(\dint x_2)\,\QQ_2(\dint M)\,\QQ_1(\dint L)\,.
\end{align*}
Applying the change of variables $u:=x_2-y_2$ and $v:=(x_2+y_2)/2$, which has Jacobian one, we get
\begin{equation}\label{eq:proofThm1LastStep}
\begin{split}
\L(A\times B\times C)= {\g_1\g_2}&\int\limits_{G(n,k_1)}\int\limits_{G(n,k_2)}[L,M]\left(\,\int\limits_U\ell_{k_1+k_2}\big(A\cap(V+v)\big)\,\ell_U(\dint v)\right)\\
&\qquad\times\frac{1}{2}\left(\,\int\limits_U{\bf 1}\left\{\|u\|\in B\right\}\mathcal{H}^0\left(L(u)\cap C\right)\,\ell_U(\dint u)\right)\,\QQ_2(\dint M)\,\QQ_1(\dint L)\,,
\end{split}
\end{equation}
where $L(u)$ stands for the one-dimensional linear subspace spanned by $u$.
Since $U=V^\perp$, the first integral in brackets is just $\ell(A)$. For the second integral we introduce spherical coordinates in the $(n-k_1-k_2)$-dimensional subspace $U=(L+M)^\perp$ and obtain
\begin{align*}
\frac{1}{2}\int\limits_U{\bf 1}\left\{\|u\|\in B\right\}\mathcal{H}^0\left(L(u)\cap C\right)\,\ell_U(\dint u) &=\frac{1}{2}
\int\limits_{\SS_U}\mathcal{H}^0\left(L(u)\cap C\right)\,\s_U(\dint u) \,
\int\limits_B t^{n-k_1-k_2-1}\,\dint t \\
&= \sigma_{U}(C\cap U)\int\limits_B t^{n-k_1-k_2-1}\,\dint t\,.
\end{align*}
Combining this with \eqref{eq:proofThm1LastStep}, we arrive at the desired expression for $\L(A\times B\times C)$. Finally, the invariance statement follows directly from this expression.
\end{proof}

The intensity and the directional distribution of $\Phi$ are denoted by $N$ and $\cR$, respectively, recall \eqref{eq:NvR}. 
As a consequence of Theorem \ref{thm:IntensityMeasureSegmentProcess}, we conclude that
$$
N=\gamma_1\gamma_2\kappa_{n-k_1-k_2}\delta^{n-k_1-k_2}\int\limits_{G(n,k_1)}\int\limits_{G(n,k_2)}[L,M]\, \QQ_2(\dint M)\, \QQ_1(\dint L)
$$
and, if $N\neq 0$,
$$
\cR(C)=\frac{\int\limits_{G(n,k_1)}\int\limits_{G(n,k_2)}\sigma_{(L+M)^\perp}(C\cap(L+M)^\perp)\,[L,M]\,\QQ_2(\dint M)\,\QQ_1(\dint L)}{\omega_{n-k_1-k_2}\int\limits_{G(n,k_1)}^{\phantom{\vspace{0.01ex}}}\int\limits_{G(n,k_2)}[L,M]\, \QQ_2(\dint M)\, \QQ_1(\dint L)}\,.
$$

We now consider the  segment process $ \Phi$ associated with a \textit{single} weakly stationary $k$-flat process $X$ of type $(S_2)$ with $1\le k<n/2$, intensity $\g$ and directional distribution $\QQ$. A slight modification of the proof of Theorem \ref{thm:IntensityMeasureSegmentProcess} shows that 
$\Phi$ is weakly stationary and yields also formulae for the proximity and the directional distribution of $\Phi$. Here, we define 
$$
\Phi(A\times B\times C)=\frac{1}{2}\sum_{(E,F)\in X_{\neq}^2} {\bf 1}\{(m(E,F),d(E,F))\in A\times B,\,[E,F]>0\}\,\frac{1}{2}\mathcal{H}^0(\phi(E,F)\cap C),
$$
where $A\subset\RR^n$, $B\subset(0,\infty)$ and $C\subset\SS^{n-1}$ are Borel sets. In this setting, we write $\pi(X,\d)$ instead of $N$ and call it the {\it proximity} of $X$. We observe the following property of $X$.

\begin{lemma}
Let $1\le k<n/2$. Then any two  flats of a weakly stationary $k$-flat process $X$ of type $(S_2)$ do not intersect each other with probability one.
\end{lemma}
\begin{proof} 
This follows from the proof of Theorem 4.4.5 (a) in \cite{SW}, which uses only the representation \eqref{eq:Intensity} of the intensity measure as well as the $(S_2)$-property of the $k$-flat process.
\end{proof}

We can now express the intensity $\pi(X,\d)$ and the directional distribution $\cR$ of the proximity process $\Phi$ in terms of the intensity $\gamma$ and the directional distribution $\QQ$ of the underlying weakly stationary $k$-flat process $X$ of type $(S_2)$. This is a generalization of the main result from \cite{Schneider99}, see also \cite[Theorem 4.4.10]{SW}.

\begin{theorem}\label{cor:Proximityk_1=k_2} Let $1\le k<n/2$, let  $X$ be a weakly stationary $k$-flat process of type $(S_2)$, and let $\d> 0$ be a given distance threshold.
\begin{itemize}
\item[\rm (a)] The intensity measure of $\Phi$ is given by
\begin{equation*}
\begin{split}
\L(A\times B\times C)= &\frac{\g^2}{2}\,\ell(A)\int\limits_B t^{n-2k-1}\,\dint t\\ &\qquad\times\int\limits_{G(n,k)}\int\limits_{G(n,k)}\,[L,M]\,\sigma_{(L+M)^\perp}(C\cap (L+M)^\perp)\,\QQ(\dint M)\,\QQ(\dint L)\,,
\end{split}
\end{equation*}
where $A\subset\RR^n$, $B\subset(0,\delta]$ and $C\subset\SS^{n-1}$ are Borel sets.
\item[\rm (b)] The proximity $\pi(X,\d)$ of $X$ equals 
 $$
 \pi(X,\d)=\frac{\g^2}{ 2}\kappa_{n-2k}\d^{n-2k}\int\limits_{G(n,k)}\int\limits_{G(n,k)}[L,M]\,\QQ(\dint M)\,\QQ(\dint L)\,.
 $$
\item[\rm (c)] If $\pi(X,\d)\neq 0$, then the directional distribution $\cR$ of the proximity process is given by 

$$
\cR(C)=\frac{\int\limits_{G(n,k)}\int\limits_{G(n,k)}[L,M]\,\sigma_{(L+M)^\perp}(C\cap (L+M)^\perp)\,\QQ(\dint M)\,\QQ(\dint L)}{\omega_{n-2k}\int\limits_{G(n,k)}^{\phantom{\vspace{0.01ex}}}\int\limits_{G(n,k)}[L,M]\,\QQ(\dint M)\,\QQ(\dint L)}\,,
$$ 
independently of $\d$, where $C\subset\SS^{n-1}$ is a 
Borel set.
\end{itemize}
\end{theorem}

\begin{proof} As above, let us denote by $\L$ the intensity measure of the (parameterized) proximity process $\Phi$ associated with the weakly stationary $k$-flat process $X$ of type $(S_2)$. To obtain a formula for $\L$, we start with
$$
\L(A\times B\times C)=\frac{1}{2}\,\EE\sum_{(E,F)\in X_{\not=}^2}{\bf 1}\left\{m(E,F)\in A,\,
d(E,F)\in B,\,[E,F]>0\right\}\frac{1}{2}\mathcal{H}^0(\phi(E,F)\cap C)\,,
$$
where $A\subset\RR^n$, $B\subset(0,\delta]$ and $C\subset\SS^{n-1}$ are Borel sets. 
The multivariate Campbell formula (see \cite[Theorem 3.1.3]{SW}) 
and the $(S_2)$-property of $X$ yield
\begin{align*}
\L(A\times B\times C)&= \frac{\g^2}{2}\int\limits_{G(n,k)}\int\limits_{G(n,k)}\int\limits_{L^\perp}\int\limits_{M^\perp}{\bf 1}\left\{ m(L+x,M+y)\in A,\,d(L+x,M+y)\in B,\,[L,M]>0\right\}\\
&\qquad\qquad\times\frac{1}{2}\mathcal{H}^0\left(\phi(L+x,M+y)\cap C\right)\,\ell_{M^\perp}(\dint y)\,\ell_{L^\perp}(\dint x)\,\QQ(\dint M)\,\QQ(\dint L)\,. 
\end{align*} 
Following now the proof of Theorem \ref{thm:IntensityMeasureSegmentProcess}, we obtain the formula in part (a). The assertions in (b) and (c) are immediate consequences.
\end{proof}

\subsection{Uniqueness problems and stability estimates}\label{sec:InvarianceAndZonoids}

Similarly as in Section \ref{sec:UniquenessHyperplanes}, we consider here the problem whether, for a weakly stationary $k$-flat process $X$ of type $(S_2)$, the directional distribution $\QQ$ of $X$ is determined by the directional distribution $\cal R$ of the proximity process $\Phi$ or, similarly, whether the intensity measure of $\Phi$ determines the intensity measure $\Theta$ of $X$.

The situation here is different from the setting discussed in Section \ref{sec:UniquenessHyperplanes}, since the flats of the intersection processes can have dimensions from $0$ to $n-2$, whereas the proximity process always consists of one-dimensional objects. Nevertheless, the theorems turn out to be quite similar and we start with a non-uniqueness result.
It applies, for instance, with $n=5$ and $k=2$. However, it does not cover the case $n=6$ and $k=2$, for example, which remains open. Notice that the case $k=n-1$ cannot occur. 

\begin{theorem}\label{nonunique}
Let $k\in\{2,\ldots,n-2\}$ be such that $2k<n<2\left(k+\left\lfloor\frac{k}{2}\right\rfloor\right)$. Then there exist two stationary Poisson processes $X_1$ and $X_2$ of $k$-flats with the same intensity $\g>0$ and  different direction distributions $\QQ_1\neq\QQ_2$, which are absolutely continuous with respect to the Haar probability measure on $G(n,k)$ and for which the associated proximity processes have the same intensity measure. 
\end{theorem}

\begin{proof}
Since $(L+M)^\perp=L^\perp\cap M^\perp$, the theorem follows from Theorem \ref{cor:Proximityk_1=k_2} (a), once we have shown that there are two probability measures measures 
$\QQ_1$, $\QQ_2$ on $G(n,k)$ with $\QQ_1\neq\QQ_2$ with continuous densities such that 
\begin{align}\label{nonunique2}
&\int\limits_{G(n,k)}\int\limits_{G(n,k)} [L,M]\,g(L^\perp\cap M^\perp)\, \QQ_1(\dint L)\, \QQ_1(\dint M)\nonumber\\
&\qquad\qquad =\int\limits_{G(n,k)}\int\limits_{G(n,k)}[L,M]\,g(L^\perp\cap M^\perp)\, \QQ_2(\dint L)\, \QQ_2(\dint M)
\end{align}
for all continuous functions $g:G(n,n-2k)\to\RR$. If $k$ satisfies the assumptions of Theorem \ref{nonunique}, then $n-k$ satisfies the assumptions 
of Theorem \ref{thmuniquedirect1}. But then relation \eqref{nonunique0} with $k$ replaced by $n-k$ and certain measures $\QQ_1'$ and $\QQ_2'$ implies \eqref{nonunique2} by passing to the image measures of these measures under the  orthogonal complement map $G(n,k)\to G(n,n-k)$, $U\mapsto U^\perp$.
\end{proof}

We now show a positive result in the case $k=1$, that is, for line processes $X$. Here, as in the hyperplane case, we identify the directional distribution $\QQ$ on $G(n,1)$ with an even probability measure on $\SS^{n-1}$. Explicitly, we identify $\QQ$ with
$$
\tilde\QQ(C)=\frac{1}{2}\int\limits_{G(n,1)}\mathcal{H}^0(U\cap C)\,\QQ(\dint U)\,,
$$
for Borel sets $C\subset\mathbb{S}^{n-1}$. 
Moreover, we assume that $\QQ$ is non-degenerate, which means that $\tilde\QQ$ is not concentrated on a great sub-sphere of $\SS^{n-1}$. Recall that $Z_\QQ$ is the zonoid associated with $\QQ$ and that $S_2(Z_\QQ,\,\cdot\,)$ is the second area measure of $Z_\QQ$. We now write $\cR_{\QQ}$ for the directional distribution of $\Phi$. 

\begin{theorem}\label{thmuniquedirect1prox}
For $n\ge 3$, let $X$ be a weakly stationary process of lines in $\RR^n$ of type $(S_2)$ with intensity $\gamma>0$ and non-degenerate directional distribution $\QQ$. Then, the associated zonoid $Z_\QQ$ is full-dimensional and the intensity 
$\pi(X,\d)$ and the directional distribution $\cR_{\QQ}$ of the proximity process $\Phi$ satisfy
\begin{equation}\label{directions}
 \pi(X,\d)=\g^2\k_{n-2}\d^{n-2}\,V_2(Z_{\QQ})\,,\qquad\cR_\QQ (\,\cdot\,)=\frac{S_2(Z_\QQ,\,\cdot\,)}{S_2(Z_\QQ,\SS^{n-1})}\,.
\end{equation}
Moreover, $\cR_\QQ$ uniquely determines $\QQ$. In particular, $\QQ$ is rotation invariant if and only if $\cR_\QQ$ is rotation invariant. 
\end{theorem}
\begin{proof} Since $\QQ$ is not concentrated on a great sub-sphere of $\SS^{n-1}$, the zonoid $Z_\QQ$ does not lie in a hyperplane, hence it is full-dimensional. 
By a special case of \cite[Theorem 8]{HTW}, the second area measure of $Z_\QQ$ is given by
$$
S_2(Z_\QQ,\,\cdot\,)=\frac{1}{2\binom{n-1}{2}}\int\limits_{\SS^{n-1}}\int\limits_{\SS^{n-1}}\nabla_2(u_1,u_2)\,
\mathcal{H}^{n-3}(\,\cdot\,\cap  u_1^\perp \cap u_2^\perp\cap\SS^{n-1})\, \QQ(\dint u_1)\,\QQ(\dint u_2)\,.
$$
Hence,  Theorem \ref{cor:Proximityk_1=k_2} (c) yields that 
$$
\cR_\QQ(\,\cdot\,)=\frac{S_2(Z_\QQ,\,\cdot\,)}{S_2(Z_\QQ,\SS^{n-1})}\,,
$$
which proves the second part of \eqref{directions}. To prove the formula for $\pi(X,\d)$, we observe that 
\begin{align*}
S_2(Z_\QQ,\SS^{n-1}) &=\frac{\o_{n-2}}{ 2\binom{n-1}{ 2}}\int\limits_{\SS^{n-1}}\int\limits_{\SS^{n-1}}\nabla_2(u_1,u_2)\, \QQ(\dint u_1)\,\QQ(\dint u_2)\\
&=\frac{\o_{n-2}}{ 2\binom{n-1}{ 2}}\int\limits_{G(n,1)}\int\limits_{G(n,1)}[L,M]\, \QQ(\dint u_1)\,\QQ(\dint u_2)\,,
\end{align*}
where we have used the identification of $\SS^{n-1}$ with $G(n,1)$ indicated above and the fact that the integrand is a symmetric function. This together with Theorem \ref{cor:Proximityk_1=k_2} (b) and the relation \eqref{eq:VmSmRelation} between $S_2(Z_\QQ,\SS^{n-1})$ and $V_2(Z_\QQ)$ imply that
\begin{equation*}
\begin{split}
\pi(X,\d) &= \frac{\g^2}{ 2}\k_{n-2}\d^{n-2}\frac{2\binom{n-1}{ 2}}{\o_{n-2}}S_2(Z_\QQ,\SS^{n-1}) = \frac{\g^2}{ 2}\k_{n-2}\d^{n-2}\frac{2\binom{n-1}{ 2}}{\o_{n-2}}\frac{n\k_{n-2}}{\binom{n}{ 2}}V_2(Z_\QQ)\\
&= \g^2\k_{n-2}\d^{n-2}\,V_2(Z_{\QQ})\,.
\end{split}
\end{equation*}
The uniqueness result can be deduced as in the proof of Theorem \ref{thmuniquedirect2}. 
\end{proof}

Theorem \ref{thmuniquedirect1prox} can be complemented by a uniqueness result for the intensity measure $\Theta$ of $X$.

\begin{corollary}
Let $n\ge 3$, and let $X$ be a weakly stationary line process of type $(S_2)$ with intensity $\gamma>0$ and non-degenerate directional distribution $\QQ$. Then the proximity $\pi(X,\delta)$  and the directional distribution $\cR_\QQ$ of the proximity process $\Phi$ uniquely determine the intensity measure $\Theta$ of $X$. In particular, if $X$ is a Poisson process, then its distribution is uniquely determined by $\pi(X,\delta)$ and $\cR_\QQ$. Moreover, the stationary Poisson line process $X$ is isotropic if and only if its proximity process $\Phi$ is isotropic.
\end{corollary}

\begin{remark}\rm 
The representation of $\pi(X,\d)$ in terms of  the intrinsic volume $V_2(Z_{\gamma\QQ})$ can be used to investigate extremum problems. We fix $\g>0$ and apply the isoperimetric-type inequality \cite[Equation (14.31)]{SW} to deduce that 
$$\pi(X,\d)\leq \frac{n-1}{2n}\frac{\k_{n-1}^2}{\k_n}\,\g^2\delta^{n-2}$$
with equality if and only if the associated zonoid $Z_\QQ$ is a ball, and therefore if and only if $\QQ$ is the uniform distribution on $\SS^{n-1}$. 
One can also study a related minimum problem. For this, we follow the strategy introduced in \cite{HS2011Zonoids} for the dual situation of intersection densities and define the affine proximity $\Pi(X,\delta)$ of a weakly stationary line process $X$ of type $(S_2)$ with respect to a distance threshold $\d>0$ as 
$$\Pi(X,\d):=\sup_{\Psi\in{\rm GL}(n)}\frac{\pi(\Psi(X),\d)}{\gamma(\Psi (X))^2}\,,$$ 
where the supremum extends over the general linear group ${\rm GL}(n)$ and where $\gamma (\Psi (X))$ is the intensity of $\Psi (X)$, the image of $X$ under $\Psi$ (it can be shown that $\Psi (X)$ is again a weakly stationary line process). Then the results of \cite{HS2011Zonoids} lead to the lower bound 
$$\Pi(X,\d)\geq \frac{n-1}{ 2n} \k_{n-2}\d^{n-2}\,$$ 
for the affine proximity $\Pi(X,\d)$. Here, equality holds if and only if $Z_\QQ$ is a parallelepiped, or equivalently, if and only if the lines of $X$ almost surely attain only $n$ fixed directions.
\end{remark}

Finally, we consider again a stability question. Given the intensities $\pi:=\pi(X,1)$ and $\pi':=\pi(X',1)$, and the directional distributions $\cR$ and $\cR'$ of two proximity processes $\Phi$ and $\Phi'$ associated with weakly stationary line processes $X$ and $X'$ in $\RR^n$. We denote by $\g$ and $\QQ$, and $\g'$ and $\QQ'$ the intensity and the directional distribution of $X$ and $X'$, respectively. Assume that $\pi\cR$ and $\pi'\cR'$ are close, are then also $\g\QQ$ and $\g'\QQ'$ close to each other? 

\begin{theorem}\label{stability3} Let $X$ and $X'$ be two weakly stationary line processes in $\RR^n$, $n\ge 3$, of type $(S_2)$ with intensities $\gamma>0$ and $\g'>0$ and directional distributions $\QQ$ and $\QQ'$, respectively. Denote by $\pi$, $\cR$  and by $\pi'$, $\cR'$ the intensity and the directional distribution of the proximity process of $X$ and $X'$, respectively, and suppose that $\pi>0$ and $\pi'>0$. 
Assume that $\g\QQ, \g'\QQ'\in\bM_{\rm e}(\rho,R)$ with some constants $0<\rho\le R<\infty$. Then there is a constant $c>0$,  depending only on $n,\rho$ and $R$, such that
$$
d_{BL}(\g\QQ,\g'\QQ') \le c \,  d_{BL}(\pi\cR,\pi'\cR')^{2c(n)}\qquad \text{and}\qquad
d_P(\g\QQ,\g'\QQ') \le c \, d_P(\pi\cR,\pi'\cR')^{c(n)}\,,
$$
where $c(n)=(2(n+1)(n+4))^{-1}$. 
\end{theorem}

\begin{proof} Equation \eqref{directions} shows that $\pi\cR$ and $S_2(Z_{\g\QQ},\,\cdot\,)$  are identical up to a multiplicative constant depending only on $n$, and the same is true for $\pi'\cR'$ and $S_2(Z_{\g'\QQ'},\,\cdot\,)$. Hence, we can argue as in the proof of 
Theorem \ref{stability2}.
\end{proof}

\begin{remark}\rm 
The assumption that $\g\QQ,\g'\QQ'\in\bM_{\rm e}(\rho,R)$ for $0<\rho\leq R<\infty$ in Theorem \ref{stability3} ensures in particular that the directional distributions $\QQ$ and $\QQ'$ are non-degenerate.
\end{remark}

\section{Limit theory for the proximity process}

We now perform a second-order analysis of certain length-power direction functionals of proximity processes. Throughout this section we assume that the underlaing $k$-flat process is a stationary Poisson $k$-flat process. We  prove central and non-central limit theorems for the resulting proximity processes. This yields results dual to those for intersection processes recently developed in \cite{LPST}. 

\subsection{A preparatory step}

Let $X$ be a stationary Poisson process of $k$-dimensional flats in $\RR^n$ such that $1\le k<n/2$ and denote by $\Theta$, $\g$ and $\QQ$, respectively, the intensity measure, the intensity and the directional distribution of $X$. We consider the stationary proximity process $\Phi$  associated with $X$ (for some fixed distance threshold $\d>0$). In this section, we are interested in limit theorems for $\Phi$ or for functionals of $\Phi$ of the form
\begin{equation*}
\begin{split}
F_\alpha(A,C) &:=\frac{1}{ 2}\sum_{(E,F)\in X_{\neq}^2}d(E,F)^\alpha\,{\bf 1}\{m(E,F)\in A,\,0<d(E,F)\leq\d,\,[E,F]>0\}\frac{1}{2}\mathcal{H}^0(\phi(E,F)\cap C)\,, 
\end{split}
\end{equation*}
where $\alpha\geq 0$ is some real-valued parameter and $A\subset\RR^n$, $C\subset \SS^{n-1}$ are Borel sets.  Recall that  $\phi(E,F)=\phi(\overline{x_Ex_F})\in G(n,1)$ stands for the direction of the perpendicular of two disjoint subspaces $E,F\in A(n,k)$ in general position. In particular, if $\alpha=0$ and $\ell(A)=1$, then $\EE F_0(A,\SS^{n-1})$ is just the proximity $\pi(X,\d)$ of $X$. If $\alpha=1$, then $\EE F_1(A,\SS^{n-1})$ is the mean total segment length within $A$. More generally, $F_0(A,C)$ or $F_1(A,C)$ are the number and the total length of segments of $\Phi$ with midpoint in $A$ and direction in $C$.

We start our investigations with the following proposition, which will be applied several times.

\begin{proposition}\label{prop:GeneralBody}
Let $A\subset\RR^n$ and $C\subset\SS^{n-1}$ be Borel sets. Then
\begin{align*} 
&\int\limits_{A(n,k)}\int\limits_{A(n,k)}d(E,F)^\a\,{\bf 1}\{m(E,F)\in A,\,0<d(E,F)\leq\d,\,[E,F]>0\}\,{1\over 2}\cH^0(\phi(E,F)\cap C)
\,\Theta(\dint E)\,\Theta(\dint F)\\
&\qquad =\g^2\frac{\d^{n-2k+\a}}{ n-2k+\a}\,\ell(A)\int\limits_{G(n,k)}\int\limits_{G(n,k)}[L,M]\,\s_{(L+M)^\perp}(C\cap (L+M)^\perp)\,\QQ(\dint M)\,\QQ(\dint L)\,.
\end{align*}
\end{proposition}

\begin{proof}
Let us denote by $\L$ the intensity measure of the proximity process of $X$. Since $X$ is a Poisson process and thus has property $(S_2)$,
\begin{align*}
&\int\limits_{A(n,k)}\int\limits_{A(n,k)}d(E,F)^\a\,{\bf 1}\{m(E,F)\in A,\,0<d(E,F)\leq\d,\,[E,F]>0\}\,{1\over 2}\cH^0(\phi(E,F)\cap C)
\,\Theta(\dint E)\,\Theta(\dint F)
\end{align*}
can be re-written as
$$\int\limits_{\RR^n\times(0,\d]\times\SS^{d-1}}d^\a\,{\bf 1}\{m\in A,\,0<d\leq\d,\,u\in C\}\,\Lambda(\dint(m,d,u))\,.$$
The result is now a direct consequence of Corollary \ref{cor:Proximityk_1=k_2} (a).
\end{proof}

\subsection{Asymptotic covariances}\label{sec:covariances}

From now on we consider the functional $F_\alpha(\,\cdot\,,\,\cdot\,)$ for a fixed intensity parameter $0<\g<\infty$ and a fixed distance threshold $\d>0$ and investigate the mean $\EE F_\alpha(A_\r,C)$ and the variance $\var(F_\alpha(A_\r,C))$ of $F_\alpha(A_\r,C)$, for a bounded Borel set $A\subset\RR^n$, a Borel set $C\subset\SS^{n-1}$, and with the scaled sets $A_\r=\r A$, $\r> 0$, as $\r\to\infty$. More generally, we fix $\alpha_1,\alpha_2\geq 0$ and Borel sets $C_1,C_2\subset\SS^{n-1}$ and study  the asymptotic covariance $\cov\big(F_{\alpha_1}(A_\r,C_1),F_{\alpha_2}(A_\r,C_2)\big)$ of $F_{\alpha_1}(A_\r,C_1)$ and $F_{\alpha_2}(A_\r,C_2)$. To describe the variance and covariance asymptotics properly, we will take $A$ from the class $\cal D$. In what follows, we write $\inter(\,\cdot\,)$, $\bd(\,\cdot\,)$, ${\rm cl}(\,\cdot\,)$ and $D(\,\cdot\,)$ for the interior, the boundary, the closure and the diameter of the argument set, respectively. Then $\cal D$ is defined as the class of all bounded Borel sets $A\subset \RR^n$ satisfying
\begin{itemize}
\item[1.] $\inter(A)\neq\emptyset$ and $\ell(\bd (A))=0$,
\item[2.] $\inter(A)=\bigcup\limits_{i\geq 1}O_i$ with countably many open convex sets $O_i$ such that $\sum\limits_{i\ge 1} D(O_i)<\infty$.
\end{itemize}
Clearly, a  set $A\in \mathcal{D}$ satisfies $\ell (A)\in (0,\infty)$. 

\begin{theorem}\label{thm:AsymptoticEandV}
Let $X$ be a stationary Poisson $k$-flat process in $\RR^n$ with $1\le k< n/2$, intensity $\g>0$ and directional distribution $\QQ$. Let $\delta>0$ be fixed. 
\begin{enumerate}
 \item [{\rm (a)}] If $\alpha\ge 0$ and $A\subset\RR^n$, $C\subset\SS^{n-1}$ are Borel sets, then
\begin{equation*}
\begin{split}
\EE F_\alpha(A_\r,C)=\frac{\g^2}{ 2}&\frac{\d^{n-2k+\a}}{ n-2k+\a}\,\r^n\,\ell(A)\\ &\times \int\limits_{G(n,k)}\int\limits_{G(n,k)}[L,M]\,\s_{(L+M)^\perp}(C\cap (L+M)^\perp)\,\QQ(\dint M)\,\QQ(\dint L)
\end{split}
\end{equation*}
for all $\varrho>0$.
 \item[{\rm (b)}] If $\alpha_1,\alpha_2\ge 0$,  $A\in\cal D$ and $C_1,C_2\subset\SS^{n-1}$ are Borel sets, then  
$$
\lim_{\r\to\infty}\frac{\cov\big(F_{\alpha_1}(A_\r,C_1),F_{\alpha_2}(A_\r,C_2)\big)}{\r^{n+k}}=\g^3
\frac{\d^{2(n-2k)+\alpha_1+\alpha_2}}{(n-2k+\alpha_1)(n-2k+\alpha_2)}\,{\cal I}(A;C_1,C_2)\,,
$$ 
where 
$$
{\cal I}(A;C_1,C_2):=\int\limits_{G(n,k)}\int\limits_{M^\perp}\ell_k(A\cap(M+y))^2\, \ell_{M^\perp}(\dint y)\,
b(M;C_1,C_2)\,\QQ(\dint M) 
$$
and $b(M;C_1,C_2):=b(M;C_1)b(M;C_2)$ with
$$
b(M;C_i):=
\int\limits_{G(n,k)}[L,M]\,\s_{(L+M)^\perp}(C_i\cap (L+M)^\perp)\,\QQ(\dint L)\, ,\quad i\in\{1,2\}\,.
$$
In the isotropic case, ${\cal I}(A;\SS^{n-1},\SS^{n-1})$ satisfies
\begin{equation}\label{eq:CPI}
{\cal I}(A;\SS^{n-1},\SS^{n-1})=\frac{\k_k}{k+1}\left((n-2k)\frac{\k_{n-k}^2}{\k_n}\frac{\binom{n-k}{k}}{\binom{n}{k}}\right)^2\int\limits_{A(n,1)}\ell_1(A\cap g)^{k+1}\,\mu_1(\dint g)\,.
\end{equation}
\end{enumerate}
\end{theorem}

\begin{remark}\rm
The parameters $\alpha_1,\alpha_2$ do not interfer with the geometry of the set $A$ in the {\em asymptotic} covariance formula in Theorem \ref{thm:AsymptoticEandV} (b). Here, $\alpha_1,\alpha_2$ only appear in the pre-factor, while the contribution of $A$ is restricted to ${\cal I}(A;C_1,C_2)$, which is independent of the parameters $\alpha_1,\alpha_2$.
\end{remark}

\begin{remark}\rm 
The integral appearing in \eqref{eq:CPI} is, up to normalization, the $(k+1)$st chord-power integral of $A$, a quantity which is well known in convex geometry, for convex bodies $A$ (see \cite{Schneider93,SW}). Relation \eqref{eq:CPI} corrects in this context the constant appearing in the asymptotic variance formula of \cite[Remark 2]{ST}. 
\end{remark}

\begin{proof}[Proof of Theorem \ref{thm:AsymptoticEandV}.]
For the proof of assertion (a), we use the fact that the Poisson process $X$ is of type $(S_2)$. This shows that  
\begin{align*}
\EE F_\alpha(A_\r,C) =\frac{1}{ 2}\int\limits_{A(n,k)}\int\limits_{A(n,k)}d(E,F)^\a\,&{\bf 1}\{m(E,F)\in A,\,0<d(E,F)\leq\d,\,[E,F]>0\}\\
& \qquad\times{1\over 2}\cH^0(\phi(E,F)\cap C)\,\Theta(\dint E)\,\Theta(\dint F)\,.
\end{align*}
Then Proposition \ref{prop:GeneralBody} yields the desired formula. 

For (b), let us introduce the functions $f_1^{(A,C_j,\d,\alpha_j)}$ on $A(n,k)$ and $f_2^{(A,C_j,\d,\alpha_j)}$ on $A(n,k)^2$, for $j=1,2$, by 
\begin{align*}
f_1^{(A,C_j,\d,\alpha_j)}(E):=\int\limits_{A(n,k)}d(E,F)^{\alpha_j}\,&{\bf 1}\{m(E,F)\in A,\,0<d(E,F)\leq\d,\,[E,F]>0\}\\
&\qquad\times{1\over 2}\cH^0(\phi(E,F)\cap C_j)\,\Theta(\dint F)\,,
\end{align*}
and 
\begin{align*}
f_2^{(A,C_j,\d,\alpha_j)}(E,F):=\frac{1}{ 2}d(E,F)^{\alpha_j}\,&{\bf 1}\{m(E,F)\in A,\,0<d(E,F)\leq\d,\,[E,F]>0\}\\
&\qquad\times{1\over 2}\cH^0(\phi(E,F)\cap C_j)\,.
\end{align*}
Then, by Theorem 1.1 in \cite{LastPenrose}, we have
\begin{equation}\label{eq:VarianceFormulaGeneral}
\cov\big(F_{\a_1}(A,C_1),F_{\a_2}(A,C_2)\big)=\langle f_1^{(A,C_1,\d,\a_1)},f_1^{(A,C_2,\d,\a_2)}\rangle_1+2\langle f_2^{(A,C_1,\d,\a_1)},f_2^{(A,C_2,\d,\a_2)}\rangle_2,
\end{equation}
where $\langle\,\cdot\,,\,\cdot\,\rangle_j$ stands for the standard scalar product in the Hilbert space $L^2(\Theta^j)$. In fact, the infinite sum in \cite{LastPenrose} only has two terms in our case. We start with the second term. The Cauchy-Schwarz inequality yields 
$$
0\leq\langle f_2^{(A,C_1,\d,\a_1)},f_2^{(A,C_2,\d,\a_2)}\rangle_2\leq\|f_2^{(A,C_1,\d,\a_1)}\|_2\,\|f_2^{(A,C_2,\d,\a_2)}\|_2\,,
$$ 
where $\|\,\cdot\,\|_2=\big(\langle\,\cdot\,,\,\cdot\,\rangle_2\big)^{1/2}$ is the norm in $L^2(\Theta^2)$. Further, using Proposition \ref{prop:GeneralBody} we find that
\begin{equation*}
\begin{split}
&\|f_2^{(A_\r,C_j,\d,\alpha_j)}\|_2^2\\
 &\le \frac{1}{ 4}\int\limits_{A(n,k)}\int\limits_{A(n,k)}d(E,F)^{2\a_j}\,{\bf 1}\{m(E,F)\in A_\r,\,0<d(E,F)\leq\d,\,[E,F]>0\}\\
 &\hspace{5cm}\times{1\over 2}\cH^0(\phi(E,F)\cap C_j)\,\Theta(\dint F)\,\Theta(\dint E)\\
&=\frac{\g^2}{ 4}\frac{\d^{n-2k+2\alpha_j}}{ n-2k+2\alpha_j}\r^n\ell(A)\int\limits_{G(n,k)}\int\limits_{G(n,k)}[L,M]\,\s_{(L+M)^\perp}(C_j\cap (L+M)^\perp)\,\QQ(\dint M)\,\QQ(\dint L)
\end{split}
\end{equation*}
for $j\in\{1,2\}$, since $( {1\over 2}\cH^0(\phi(E,F)\cap C_j))^2\le  {1\over 2}\cH^0(\phi(E,F)\cap C_j)$. Consequently, there is a constant $c>0$, depending only on $A$, $C_j$, $\alpha_j$, $\d$ and $\g$ such that 
$$
0\leq\lim_{\r\to\infty}\frac{2\langle f_2^{(A_\varrho,C_1,\d,\a_1)},f_2^{(A_\varrho,C_2,\d,\a_2)}\rangle_2}{\r^{n+k}}\leq\lim_{\r\to\infty}\frac{c\,\r^n}{\r^{n+k}}=0\,,
$$ 
since $k\geq 1$. We now consider the first term in \eqref{eq:VarianceFormulaGeneral}. Similarly as in the proof of Lemma 4.1 in \cite{ST}, 
it follows that
\begin{equation*}
\begin{split}
f_1^{(A,C_j,\d,\a_j)}(M+y)=\g\int\limits_{G(n,k)}\int\limits_{(L+M)^\perp}[L,M]\,&\|x\|^{\alpha_j}\,\ell_k\left(A\cap\left(M+\frac{x}{ 2}+y\right)\right)\\
&\times{\bf 1}\{0<\|x\|\leq\d,\,x/\|x\|\in C_j\}\,\ell_{(L+M)^\perp}(\dint x)\,\QQ(\dint L)
\end{split}
\end{equation*}
for $j\in\{1,2\}$ and $M\in G(n,k)$, $y\in M^\perp$. 
From this we conclude that 
$$
f_1^{(A_\r,C_j,\d,\a_j)}(M+y)=\r^{n-k+\alpha_j}f_1^{(A,C_j,\d/\r,\a_j)}(M+y/\r)\,,\qquad j\in\{1,2\}\,.
$$ 
This yields
\begin{equation*}
\begin{split}
\langle f_1^{(A_\varrho,C_1,\d,\a_1)},f_1^{(A_\varrho,C_2,\d,\a_2)}\rangle_1 &= \r^{2(n-k)+\a_1+\a_2}\,\g\int\limits_{G(n,k)}\int\limits_{M^\perp}f_1^{(A,C_1,\d/\r,\a_1)}(M+y/\r)\\ &\hspace{5cm}\times\,f_1^{(A,C_2,\d/\r,\a_2)}(M+y/\r)\,\ell_{M^\perp}(\dint y)\, \QQ(\dint M)\\
&=\r^{3(n-k)+\alpha_1+\a_2}\,\g\int\limits_{G(n,k)}\int\limits_{M^\perp}f_1^{(A,C_1,\d/\r,\a_1)}(M+y)\\ &\hspace{5cm}\times\,f_1^{(A,C_2,\d/\r,\a_2)}(M+y)\,\ell_{M^\perp}(\dint y)\,\QQ(\dint M)\\
&=\r^{3(n-k)+\alpha_1+\a_2}\langle f_1^{(A,C_1,\d/\r,\a_1)},f_1^{(A,C_2,\d/\r,\a_2)}\rangle_1\,.
\end{split}
\end{equation*}
For $j\in\{1,2\}$, we now consider
\begin{equation*}
\begin{split}
&\r^{n-2k+\alpha_j}f_1^{(A,C_j,\d/\r,\a_j)}(M+y)\\
&= \r^{n-2k+\alpha_j}\,\g\int\limits_{G(n,k)}\int\limits_{(L+M)^\perp}[L,M]\,\|x\|^{\alpha_j}\,\ell_k\left(A\cap\left(M+\frac{x}{ 2}+y\right)\right)\\
&\qquad\qquad\qquad\qquad \times {\bf 1}\{0<\|x\|\leq\d/\r,\,x/\|x\|\in C_j\}\,\ell_{(L+M)^\perp}(\dint x)\,\QQ(\dint L)\\
&=\g\int\limits_{G(n,k)}\int\limits_{(L+M)^\perp}[L,M]\,\|x\|^{\alpha_j}\,\ell_k\left(A\cap\left(M+\frac{x}{2\r}+y\right)\right)\\
&\qquad\qquad\qquad\qquad\times  {\bf 1}\{0<\|x\|\leq\d,\,x/\|x\|\in C_j\}\,\ell_{(L+M)^\perp}(\dint x)\,\QQ(\dint L)\, .
\end{split}
\end{equation*}
Here, for given $M$ and $\ell_{M^\bot}$-almost all $y$, $A$ can be replaced by $\inter (A)$. In fact, by   Fubini's Theorem 
\begin{equation*}
\begin{split}
&\r^{n-2k+\alpha_j}\int_{M^\perp}f_1^{(A,C_j,\d/\r,\a_j)}(M+y)\ell_{M^\bot}(\dint y)\\
&=\g\int\limits_{G(n,k)}\int\limits_{(L+M)^\perp}[L,M]\,\|x\|^{\alpha_j}\,\int_{M^\perp}\ell_k\left(A\cap\left(M+\frac{x}{2\r}+y\right)\right)\ell_{M^\bot}(\dint y)\\
&\qquad\qquad\qquad\qquad\times  {\bf 1}\{0<\|x\|\leq\d,\,x/\|x\|\in C_j\}\,\ell_{(L+M)^\perp}(\dint x)\,\QQ(\dint L)\\
&=\ell (A) \, \g\int\limits_{G(n,k)}\int\limits_{(L+M)^\perp}[L,M]\,\|x\|^{\alpha_j}\,{\bf 1}\{0<\|x\|\leq\d,\,x/\|x\|\in C_j\}\,\ell_{(L+M)^\perp}(\dint x)\,\QQ(\dint L)
\end{split}
\end{equation*}
and, similarly,
\begin{equation*}
\begin{split}
&\r^{n-2k+\alpha_j}\int_{M^\perp}f_1^{(\inter (A),C_j,\d/\r,\a_j)}(M+y)\ell_{M^\bot}(\dint y)\\
&=\ell (\inter(A)) \, \g\int\limits_{G(n,k)}\int\limits_{(L+M)^\perp}[L,M]\,\|x\|^{\alpha_j}\,{\bf 1}\{0<\|x\|\leq\d,\,x/\|x\|\in C_j\}\,\ell_{(L+M)^\perp}(\dint x)\,\QQ(\dint L)\, .
\end{split}
\end{equation*} 
Now the assertion follows from  $\ell(\bd (A))=0$, which holds for $A\in\mathcal{D}$. We therefore assume now that $A$ is open. By neglecting a set of translation vectors $y\in M^\perp$ of  measure 0 (we have to exclude vectors $y$ in the boundary  of the orthogonal projection of one of the sets $O_i$ to $M^\perp$), we can further assume that $M+y$ does not touch any of the sets $\text{cl}(O_i)$ in the representation $A=\bigcup_{i\ge 1} O_i$. Set $E=M+y$, and let $v\in L(E)^\perp$. Then
\begin{align*}
\big|\ell_k(A\cap(E+v))-\ell_k(A\cap E)\big|&=\big|\ell_k((A-v)\cap E)-\ell_k(A\cap E)\big|\\
&=\ell_k\big(((A-v)\setminus A) \cap E\big)+\ell_k\big((A\setminus(A-v))\cap E\big)\\
&\le\sum_{i\ge 1}\big(\ell_k(((O_i-v)\setminus O_i) \cap E)+\ell_k((O_i\setminus(O_i-v))\cap E)\big)\\
&=\sum_{i\ge 1}\big|\ell_k((O_i-v)\cap E)-\ell_k( O_i\cap E)\big|\,.
\end{align*}
In the last sum, $O_i$ can be replaced by $\text{cl}(O_i)$, $i\ge 1$. But then Theorem 1.8.10 and Theorem 1.8.20 in \cite{Schneider93} can be applied (with $E$ replaced by the intersection of $E$ with a sufficiently large convex body) to see that $\ell_k((\text{cl}(O_i)-v)\cap E)\to \ell_k(\text{cl}(O_i)\cap E)$ as $\|v\|\to 0$, for each $i\ge 1$. Since 
$$
\sum_{i\ge 1}\ell_k((\text{cl}(O_i)-v)\cap E) \le \kappa_k\sum_{i\ge 1}D(O_i)^k<\infty\,,
$$ 
we conclude that $\ell_k(A\cap(E+v))\to \ell_k(A\cap E)$, as $\|v\|\to 0$. 

By assumption, $A$ has finite diameter $D(A)$. Hence, since 
\begin{equation}\label{domconv}
\ell_k\left(A\cap\left(M+\frac{x}{2\varrho}+y\right)\right)\le \kappa_k\,D(A)^k\,{\bf 1}\left\{y\in (A+\delta/(2\varrho)B^n)|M^\perp\right\}\,,
\end{equation}
the dominated convergence theorem can be applied and we obtain
\begin{equation}
\begin{split}
&\r^{n-2k+\alpha_j}f_1^{(\inter (A),C_j,\d/\r,\a_j)}(M+y)\label{pointconv}\\
&\longrightarrow \g\int\limits_{G(n,k)}\int\limits_{(L+M)^\perp}[L,M]\,\|x\|^{\alpha_j}\ell_k(A\cap (M+y))\\
&\qquad\qquad\qquad\qquad\times {\bf 1}\{0<\|x\|\leq\d,\,x/\|x\|\in  C_j\}\,\ell_{(L+M)^\perp}(\dint x)\,\QQ(\dint L)\\
&= \g\frac{\d^{n-2k+\alpha_j}}{ n-2k+\alpha_j}\,\ell_k(A\cap(M+y))\int\limits_{G(n,k)}[L,M]\,\s_{(L+M)^\perp}(C_j\cap (L+M)^\perp)\,\QQ(\dint L)
\end{split}
\end{equation}
as $\r\to\infty$. Using the pointwise convergence \eqref{pointconv} and again \eqref{domconv}, we further deduce from the dominated convergence theorem that
\begin{equation*}
\begin{split}
\frac{\langle f_1^{(A_\varrho,C_1,\d,\a_1)},f_1^{(A_\varrho,C_2,\d,\a_2)}\rangle_1}{\r^{n+k}} &= \frac{\r^{3(n-k)+\alpha_1+\a_2}\langle f_1^{(A,C_1,\d/\r,\a_1)},f_1^{(A,C_2,\d/\r,\a_2)}\rangle_1}{\r^{n+k}}\\ &= \r^{2(n-2k)+\alpha_1+\a_2}\langle f_1^{(A,C_1,\d/\r,\a_1)},f_1^{(A,C_2,\d/\r,\a_2)}\rangle_1\\ &\longrightarrow \g^3\frac{\d^{2(n-2k)+\alpha_1+\a_2}}{ (n-2k+\alpha_1)(n-2k+\a_2)}
\,{\cal I}(A;C_1,C_2)\,,
\end{split}
\end{equation*}
as $\r\to\infty$. 

It remains to consider the isotropic case. Here, we use that
$$
\int\limits_{G(n,k)}[L,M]\, \nu_k(dL)=c(n,k,k)\,,
$$
with $c(n,k,k)$ as in \eqref{eq:Defcnk1k2}. Thus, $b(M;\SS^{n-1},\SS^{n-1})=b(M;\SS^{n-1})^2$ with
$$b(M;\SS^{n-1})=\o_{n-2k}\int\limits_{G(n,k)}[L,M]\, \nu_k(dL)=(n-2k)\frac{\kappa_{n-k}^2}{\kappa_n}\frac{\binom{n-k}{k}}{\binom{n}{k}}\,.$$
For the final reformulation, Equation (8.57) in \cite{SW} can be applied (although the result in \cite{SW} is formulated for convex bodies, an inspection of the proof shows that it remains valid for arbitrary Borel sets).
\end{proof}

\subsection{Central limit theorems}

After having investigated first- and second-order properties of length-power direction functionals $F_\a(A,C)$ of proximity processes, we now turn to the central limit problem. For this, recall the definition of the set class $\mathcal{D}$ and fix $A\in{\cal D}$ and a Borel set $C\subset\SS^{n-1}$. For $\alpha\geq 0$, we define the normalized random variable 
$$
\bar F_\alpha(A_\r,C):=\varrho^{-\frac{(n+k)}{ 2}}\big(F_\alpha(A_\r,C)-\EE F_\alpha(A_\r,C)\big)\,.
$$ 
In part (a) of the next theorem we derive a univariate central limit theorem for $\bar F(A_\r,C)$, while a multivariate central limit theorem for the normalized random vector 
$$
\bar{\bf F}_{\mathbf \alpha}(A_\r,{\bf C)}:=\r^{-\frac{n+k}{ 2}}\big(F_{\alpha_1}(A_\r,C_1)-\EE F_{\alpha_1}(A_\r,C_1),\ldots,F_{\alpha_m}(A_\r,C_m)-\EE F_{\alpha_m}(A_\r,C_m)\big)\,,
$$ where $m\in\NN$ is a fixed integer, $\alpha_1,\ldots,\alpha_m\geq 0$, and $C_1,\ldots,C_m\subset\SS^{n-1}$ are Borel sets, is the content of part (b). The bold face letter ${\bf C}$ represents tuples of these data, i.e. ${\bf C}=(C_1,\ldots,C_m)$. 

\begin{theorem}\label{thm:CLT}
Consider a stationary Poisson $k$-flat process of intensity $\g>0$ with $1\le k<n/2$.
\begin{enumerate}
 \item[\rm (a)] The random variable $\bar F_\alpha(A_\r,C)$ converges in distribution, as $\r\to\infty$, to a centred Gaussian random variable $N(\sigma^2(A,C))$ with variance $\s^2(A,C)$ given by 
\begin{equation}\label{eq:CLTAsyVar}
\s^2(A,C)=\g^3\frac{\d^{2(n-2k+\alpha)}}{ (n-2k+\alpha)^2}  \,{\cal I}(A;C,C)\,.
\end{equation}
 \item[\rm (b)] The random vector $\bar{\bf F}_{\mathbf{\alpha}}(A_\r,\mathbf{C})$ converges in distribution, as $\r\to\infty$, to a centred Gaussian random vector ${\bf N}(\Sigma)$ with covariance matrix $\Sigma=(\sigma_{ij})_{i,j=1}^m$ given by 
\begin{equation}\label{eq:CLTAsyCov}
\sigma_{ij}=\g^3 \frac{ \d^{2(n-2k)+\alpha_i+\alpha_j}   }{(n-2k+\alpha_i)(n-2k+\alpha_j)}\,{\cal I}(A;C_i,C_j)\,.
\end{equation}
\end{enumerate}
\end{theorem}
\begin{proof}
To prove the theorem, we use the method of cumulants, an equivalent of the classical method of moments for which we refer to \cite[Section 30]{Billingsley}. To apply it, recall that the joint cumulant of real-valued random variables $Y_1,\ldots,Y_m$, $m\geq 1$, possessing moments of all orders, is given by
$$
\Gamma(Y_1,\ldots,Y_m) := (-\mathfrak{i})^m\frac{\partial^m}{\partial z_1\ldots\partial z_m}\log\EE \exp(\mathfrak{i}\,\lan { z},{\bf Y}\ran)\,\Big|_{{ z}={ 0}}\,,
$$
where $\mathfrak{i}$ is the imaginary unit, ${\bf Y}=(Y_1,\ldots,Y_m)$, ${ z}=(z_1,\ldots,z_m)\in\RR^m$ and ${ 0}$ stands for the zero-vector in $\RR^m$. Moreover, for integers $k\geq 1$, the $k$th cumulant of a single random variable $Y$ having finite moments of all orders is given by $\Gamma_k(Y):=\G(Y,\ldots,Y)$, where the entry $Y$ appears exactly $k$ times. According to \cite[page 191]{ProkhorovStatu} it is known that a family of centred random variables $(Y_\varrho)_{\varrho>0}$ converges in distribution to a centred Gaussian random variable with variance $\sigma^2>0$, as $\varrho\to\infty$, if
\begin{equation}\label{eq:CLTProof1}
\lim_{\varrho\to\infty}\var(Y_\varrho)=\sigma^2\qquad\text{and}\qquad\lim_{\varrho\to\infty}\G_k(Y_\varrho)=0\quad\text{for all }k\geq 3\,.
\end{equation}
Similarly, a family $({\bf Y}_\varrho)_{\varrho\geq 1}$ of centred random vectors in $\RR^m$ converges in distribution to a centred Gaussian random vector with covariance matrix $\Sigma=(\sigma_{ij})_{i,j=1}^m$, as $\varrho\to\infty$, provided that
\begin{equation}\label{eq:CLTProof2}
\begin{split}
&\lim_{\varrho\to\infty}\cov(Y_\varrho^{(i)},Y_\varrho^{(j)})=\sigma_{ij}^2\qquad\text{for all } 1\leq i,j\leq m\,,\\
&\lim_{\varrho\to\infty}\G(Y_\varrho^{(i_1)},\ldots,Y_\varrho^{(i_k)})=0\quad \text{for all }k\geq 3\text{ and }1\leq i_1\leq\ldots\leq i_k\leq m\,,
\end{split}
\end{equation}
where $Y_\varrho^{(i)}$ stands for the $i$th coordinate of ${\bf Y}_\varrho$, $1\leq i\leq m$.

A functional of the type $F_\a(A,C)$, with $\a\geq 0$, $A\in{\cal D}$ and a Borel set $C\subset\SS^{n-1}$, is a special case of a so-called Poisson U-statistic of order two as introduced in \cite{RS}. For such functionals, explicit formulae for cumulants and joint cumulants have been developed in \cite{LPST} (notice that due to our assumptions, the random variables we consider have finite moments of all orders). Although most of the theory in \cite{LPST} deals with an increasing intensity parameter $\gamma$ while (in our language) $\varrho$ is kept fixed, scaling properties of Poisson $k$-flat processes imply that the results are also applicable for fixed $\gamma$ and increasing window size $\varrho$. In fact, the intensity parameter -- called $t$ in \cite{LPST} -- has to be replaced by $\varrho^{1/(n-k)}$. In this set-up, Corollary 3 in \cite{LPST} 
 shows that \eqref{eq:CLTProof1} is satisfied for the family of random variables $(\bar F_\alpha(A_\varrho,C))_{\varrho>0}$ with asymptotic variance as in \eqref{eq:CLTAsyVar}. In the multivariate case, we can apply \cite[Theorem 2]{LPST}  to conclude that \eqref{eq:CLTProof2} holds for the family $(\bar{\bf F}_\alpha(A_\varrho,\bC))_{\varrho>0}$ of $m$-dimensional random vectors with asymptotic covariances $\sigma_{ij}^2$ given by \eqref{eq:CLTAsyCov}. This completes the proof.
\end{proof}

\begin{remark}\rm
If $\Sigma$ is the covariance matrix from Theorem \ref{thm:CLT} (b), then the associated quadratic form satisfies
$$
\langle z, \Sigma z\rangle = \g^3 \d^{2(n-2k)}\int\limits_{G(n,k)}\ell_k(A;M)\langle z,{a}{b}(M)\rangle^2\, \QQ(\dint M)\,,\qquad z\in\RR^m\,,
$$
with $a_i:=\d^\alpha_i/(n-2k+\alpha_i)$, and with ${a}:=(a_1,\ldots,a_m)$, $b(M):=(b(M;C_1),\ldots,b(M;C_m))$, 
${a}{b}(M):=(a_1b(M;C_1),\ldots,a_mb(M;C_m))$, and 
$$
\ell_k(A;M):=\int\limits_{M^\perp}\ell_k(A\cap (M+y))^2\,\ell_{M^\perp}(\dint y)>0\,,\qquad M\in G(n,k)\,,
$$
since $\ell(A)>0$ by our assumption that $A\in{\cal D}$. 
If $C_i=C$ for $i=1,\ldots,m$, then $\langle z,{a}{b}(M)\rangle=\langle z,{a}\rangle b(M;C)$, and therefore $\Sigma$ has rank $1$ 
if $b(\,\cdot\,;C)\neq 0$ on a set of positive $\QQ$-measure. On the other hand, if there exist $m$ subspaces $M_1,\ldots,M_m$ in the support of $\QQ$ 
such that the vectors ${a}{b}(M_1),\ldots,{a}{b}(M_m)$ are linearly independent, then $\Sigma$ is positive definite.
\end{remark}

\begin{remark}\rm
One can refine Theorem \ref{thm:CLT} by providing rates of convergence. In the univariate case, for example, the results in \cite{EichelsbacherTh} (see in particular Example 4.12 there) can be applied to deduce that 
$$\sup_{t\in\RR}\Big|\PP\Big(\frac{F_\a(A_\r,C)-\EE F_\a(A_\r,C)}{\sqrt{\var(F_\a(A_\r,C))}}\leq t\Big)-\PP(N\leq t)\Big|\leq c\,\r^{-(n-k)/2}$$ 
for a constant $c>0$ only depending on $A$, $C$, $\a$ and $n$ and where $N$ stands for a standard Gaussian random variable. In the multivariate case, the bounds obtained in \cite{PeZheng10} can be applied (see \cite{LPST} for related results).
\end{remark}

\subsection{Limit theorems for order statistics and related point processes}\label{subsec:PPPConvergence}

We assume the same set-up as in the previous subsections, that is, $X$ is a stationary Poisson process of $k$-dimensional flats in $\RR^n$ with intensity $0<\g<\infty$ and directional distribution $\QQ$ satisfying $1\le k<n/2$. In this section, we are not interested in limit theorems for cumulative functionals of the proximity process, but  in limit theorems related to order statistics of length-power direction functionals. More precisely, let us fix some $\alpha>0$, a Borel bounded set $A\subset\RR^n$ with $\ell(A)>0$ and a Borel set $C\subset\SS^{n-1}$. For $\varrho\geq 1$ and $\d>0$, we define the point process $\xi_\varrho$ on the positive real half-line $\RR_+$ by 
$$
\xi_\varrho:=\frac{1}{ 2}\sum_{(E,F)\in X_{\neq}^2}\delta_{d(E,F)^\alpha}\,{\bf 1}\{m(E,F)\in A_\varrho,\,d(E,F)\leq\d,\,[E,F]>0\}\,{1\over 2}\cH^0(\phi(E,F)\cap C)\,.
$$ 
Let $F_\varrho^{(m)}$, for $m\in\NN$, denote the distance from the origin to the $m$th largest point of $\xi_\varrho$. In particular, for $m=1$ and $\a=1$, this is the smallest distance between two different $k$-flats from $X$ in general position with midpoint in $A_\varrho$ and direction of the orthogonal line segment in $C$. Our aim is to describe the limit behavior of the suitably re-scaled point process $\xi_\varrho$ and of the order-statistics $F_\varrho^{(m)}$, as $\varrho$ tends to infinity. We will see that the limits are independent of the choice of the distance threshold $\d$, which is thus suppressed in our notation. The following result generalizes parts of the theory developed in \cite{ST} to arbitrary length-powers and also allows for directional constraints.

\begin{theorem}\label{thm:PPPConvergence}
Let $X$ be a stationary Poisson $k$-flat process in $\RR^n$ with $1\leq k\leq n/2$ with intensity $0<\g<\infty$ and directional distribution $\QQ$. For a bounded Borel set $A\subset\RR^n$ with $\ell(A)>0$ and a Borel set $C\subset\SS^{n-1}$ define 
$$
\beta:=\frac{\g^2}{ 2(n-2k)}\,\ell(A)\int\limits_{G(n,k)}\int\limits_{G(n,k)}[L,M]\,\s_{(L+M)^\perp}(C\cap (L+M)^\perp)\,\QQ(\dint M)\,\QQ(\dint L)\,.
$$ 
Then $\varrho^{n\alpha/(n-2k)}\xi_\varrho$ converges in distribution to a Poisson point process on $\RR_+$ with intensity measure given by
\begin{equation}\label{eq:IntensityMeasurePPP}
B\mapsto (n-2k)\,\frac{\beta}{ \alpha}\,\int\limits_B u^{(n-2k-\alpha)/\alpha}\,\dint u\,,
\end{equation}
where $B\subset\RR_+$ is a Borel set. In particular, $\varrho^{n\alpha/(n-2k)}F_\varrho^{(1)}$ converges in distribution to a Weibull random variable with distribution function $x\mapsto 1-\exp(-\beta x^{(n-2k)/\alpha})\,{\bf 1}\{x>0\}$. 
\end{theorem}

\begin{proof}
Our aim is to apply Theorem 1.1 in \cite{STScaling} (or its reformulation as Proposition 2 in \cite{ST}), which yields the result once conditions (4) and (5) in \cite{STScaling} (or the first and the second part of condition (25) in \cite{ST}) are fulfilled. For this, we define for $x>0$ the function $q_\varrho(x)$ by 
\begin{align*}
q_\varrho(x)=\frac{1}{ 2}\,\EE\sum_{(E,F)\in X_{\neq}^2}&{\bf 1}\{0<d(E,F)\leq x^{1/\a}\varrho^{-n/(n-2k)}\wedge\d^{1/\a}\}\\
&\qquad\times{\bf 1}\{ m(E,F)\in A_\varrho,\,[E,F]>0\}\,{1\over 2}\cH^0(\phi(E,F)\cap C)\,.
\end{align*}
Now, observe that there exists $\varrho_0>0$, depending on $x$, $\delta$, $n$, $k$ and $\alpha$, such that $x^{1/\a}\varrho^{-n/(n-2k)}\wedge\d^{1/\a}=x^{1/\a}\varrho^{-n/(n-2k)}$ for all $\varrho\geq\varrho_0$. Using the fact that the Poisson process $X$ is of type $(S_2)$ and applying Proposition \ref{prop:GeneralBody} with $\delta=x^{1/\alpha}\varrho^{-n/(n-2k)}$, we see that 
\begin{equation*}
\begin{split}
q_\varrho(x) &=\frac{\g^2}{ 2(n-2k)}{\ell(A_\varrho)\,\varrho^{-n}\, x^{(n-2k)/\alpha}}\int\limits_{G(n,k)}\int\limits_{G(n,k)}[L,M]\,\s_{(L+M)^\perp}(C\cap (L+M)^\perp)\,\QQ(\dint M)\,\QQ(\dint L)\\
&= \beta\,x^{(n-2k)/\alpha}
\end{split}
\end{equation*}
for all $\varrho\geq\varrho_0$. Especially, $\lim\limits_{\varrho\to\infty}q_\varrho(x)= \beta\,x^{(n-2k)/\alpha}$, which is condition (4) of \cite[Theorem 1.1]{STScaling} (or, equivalently, the first part of condition (25) in \cite{ST}). Next, we define for $x>0$ the function $r_\varrho(x)$ by 
\begin{equation*}
\begin{split}
r_\varrho(x)=\sup_{\stackrel{L\in G(n,k)}{y_L\in L^\perp}}&\g\int\limits_{G(n,k)}\int\limits_{M^\perp} \,{\bf 1}\{0<d(L+y_L,M+y_M)\leq x^{1/\a}\varrho^{-n/(n-2k)}\wedge\d^{1/\a}\}\\
&\times{\bf 1}\{m(L+y_L,M+y_M)\in A_\varrho,\,[L,M]>0\}\,{1\over 2}\cH^0(\phi(L,M)\cap C)\,\ell_{M^\perp}(\dint y_M)\,\QQ(\dint M)\,.
\end{split}
\end{equation*}
Denote by $R(A)$ the radius of the smallest ball containing $A$ and note that $R(A_\r)=\r R(A)$. Using this together with the fact that $[L,M]\leq 1$, and arguing as in the proof of Theorem \ref{thm:IntensityMeasureSegmentProcess}, we see that, for $\varrho\geq\varrho_0$,
\begin{equation*}
\begin{split}
r_\varrho(x) & = \sup_{\stackrel{L\in G(n,k)}{y_L\in L^\perp}}\g\int\limits_{G(n,k)}\int\limits_{(L+M)^\perp} [L,M]\, \ell_k((A_\r-y_L-(y_M/2))\cap L)\\
&\qquad\qquad\qquad\times {\bf 1}\{0<\|y_M\|\leq x^{1/\a}\varrho^{-2/(n-2k)}\}\,\ell_{(L+M)^\perp}(\dint y_M)\,\QQ(\dint M)\\
&\leq \sup_{\stackrel{L\in G(n,k)}{y_L\in L^\perp}}\g\int\limits_{G(n,k)}\int\limits_{(L+M)^\perp} [L,M]\,\k_k\,R(A_\varrho)^k\\
&\qquad\qquad\qquad\times {\bf 1}\{0<\|y_M\|\leq x^{1/\alpha}\varrho^{-n/(n-2k)}\}\,\ell_{(L+M)^\perp}(\dint y_M)\,\QQ(\dint M)\\
&\leq\k_k\k_{n-2k}\,R(A)^k\,x^{(n-2k)/\alpha}\,\varrho^{-(n-k)}\,,
\end{split}
\end{equation*}
where we have passed to spherical coordinates in $(L+M)^\perp$ in the last step (see also the proof of Theorem 3 in \cite{ST}). Hence, $r_\varrho(x)\to 0$, as $\varrho\to\infty$ for all $x>0$, since $A$ is bounded and thus satisfies $R(A)<\infty$. This is condition (5) of \cite[Theorem 1.1]{STScaling} (or the second part of condition (25) in \cite{ST}), which implies the convergence of $\varrho^{n\alpha/(n-2k)}\xi_\varrho$ to the limiting Poisson point process on $\RR_+$ with intensity measure given by \eqref{eq:IntensityMeasurePPP}. This completes the proof.
\end{proof}

\begin{remark}\rm
One can provide an upper bound for the rate of convergence of $F_\varrho^{(m)}$ to its limit. Indeed, using the notation from above, \cite[Theorem 1.1]{STScaling} (or Proposition 2 in \cite{ST}) implies that an upper bound for $$\Big|\PP(\varrho^{n\alpha/(n-2k)}F_\varrho^{(m)}>x)- \exp\big({-\beta x^{(n-2k)/\alpha}}\big)\sum_{j=0}^{m-1}\frac{(\beta x^{(n-2k)/\alpha})^j}{ j!}\Big|$$ is given by $\sqrt{r_\varrho(x)}=c\,\varrho^{-(n-k)/2}$ for any $\varrho\geq\varrho_0$ and $m\in\NN$, where $c>0$ is a constant only depending on $n$, $k$, $\alpha$, $A$, $C$ and $x$.
\end{remark}

\appendix
\section{Appendix}

We present here a construction of an ordinary point process as well as of a $k$-flat process of type $(S_r)$, for an arbitrary integer $r\geq 2$, which is weakly stationary and not a Poisson process. Our construction refines and extends that provided in \cite{BaddeleySilverman}.  

The point processes will be constructed in $d$-dimensional Euclidean space $\RR^d$, $d\geq 1$. Let $\{Q_i\}_{i\geq 1}$ be a dissection of $\RR^d$ into cubes of unit volume and let $\{N_i\}_{i\geq 1}$ be a sequence of independent random variables with values in $\NN_0=\{0,1,2,\ldots\}$ and common distribution. In what follows, we denote by $N$ a generic random variable following the same law and we assume that $\EE N$ is finite. For each 
$i\in\NN$, let $\{X^i_j\}_{j\geq 1}$ be a sequence of independent random variables, which are uniformly distributed in $Q_i$ and such that all the random variables 
$\{X^i_j\}_{i,j\geq 1}$ and $\{N_i\}_{i\geq 1}$ are independent. Then we define the point process $Y$ in $\RR^d$ by 
\begin{equation}\label{eq:defPPX}
Y:=\sum_{i=1}^\infty\sum_{j=1}^{N_i}\delta_{X^i_j}\,.
\end{equation}
Note that, by construction, $Y$ is not stationary if $N$ has finite support. However, $Y$ is weakly stationary. To see this, let $A\subset\RR^d$ be a Borel set and observe that
\begin{align*}
\EE Y(A)&=  \EE\sum_{i= 1}^\infty\sum_{j=1}^{N_i}{\bf 1}\{X^i_j\in A\}
=\sum_{i= 1}^\infty\sum_{m=0}^\infty\PP(N=m)\EE\sum_{j=1}^{m}{\bf 1}\{X^i_j\in A\}\\
&=\sum_{i= 1}^\infty\sum_{m=0}^\infty m\,\PP(N=m)\,\ell_d\left(A\cap Q_i\right) 
=\sum_{m=0}^\infty m\,\PP(N=m)\sum_{i= 1}^\infty\ell_d\left(A\cap Q_i\right)\\
&=(\EE N)\,\ell_d(A)\,.
\end{align*}
Let us recall from \cite[Theorem 4.4]{Heinrich} the fact that a point process is a Poisson process if and only if it is of type $(S_r)$ for all integers $r\geq 2$. Moreover, $Y$ is a Poisson process if and only if the random variable $N$ has a Poisson distribution.

\begin{proposition}\label{thm:App1}
Let $\ka\geq 2$ be an integer and let the random variable $N$ be concentrated on $\{0,1,\ldots,\ka\}$ and such that
\begin{equation*}
\EE\left[N(N-1)\cdots(N-m+1)\right]=\left(\EE N\right)^m, \qquad\text{for all}\ m\in\{1,\ldots,\ka\}\,.
\end{equation*}
Then the point process $Y$ is of type $(S_r)$ for $r=2,\ldots,\ka$, but not for $r=\ka+1$. In particular, $Y$ is not a Poisson process.
\end{proposition}

\begin{proof}
For fixed $\ka\ge 2$, we prove the claim by induction with respect to $r\in\{1,\ldots,\ka\}$. For $r=1$ there is nothing to show. Now, assume that $r\ge 2$ and that the assertion has already been proved up to $r-1$. It is sufficient to show that 
\begin{equation}\label{sr}
\EE Y^{r}_{\neq}=\big(\EE Y\big)^r,
\end{equation}
where  (as introduced in Section \ref{sec:preliminaries})
$$Y^{r}_{\neq}(\,\cdot\,)=\sum_{x_1\in Y}\ldots\sum_{x_{r-1}\in Y\setminus\{x_1,\ldots,x_{r-2}\}}\ \sum_{x_{r}\in Y\setminus\{x_1,\ldots,x_{r-1}\}}\ind\{(x_1,\ldots,x_r)\in\,\cdot\,\}\,,$$
$\EE Y^{r}_{\neq}$ is the $r$th factorial moment measure and $\EE Y$ is the intensity measure of $Y$. 

It is sufficient to verify \eqref{sr} on sets of the form $A_1\times\ldots\times A_r$, where the Borel set $A_j$ lies in the interior of a uniquely determined cube $Q_{l(j)}$, $j=1,\ldots ,r$. Here we use the fact that the  random points $X^i_j$ are distributed uniformly in $Q_i$ and hence almost surely do not lie on the boundary of $Q_i$. If 
\eqref{sr} is established for sets as described before, we can use $\sigma$-additivity in each argument to obtain the equality for arbitrary products of measurable sets. The extension to arbitrary measurable sets then follows from the usual measure extension arguments, since all measures are $\sigma$-finite.

Now, to establish the asserted equality, we distinguish two cases.

\medskip

\textit{Case I.} If $l(r)\notin\{l(1),\ldots,l(r-1)\}$, then
\begin{align*}
&Y^r_{\neq}(A_1\times\ldots\times A_r)\\
&\qquad =\sum_{x_1\in Y}\ldots\sum_{x_{r-1}\in Y\setminus\{x_1,\ldots,x_{r-2}\}}\ \sum_{x_{r}\in Y\setminus\{x_1,\ldots,x_{r-1}\}}\ind\{x_1\in A_1,\ldots,x_r\in A_r\}\\
&\qquad =\sum_{x_1\in Y}\ldots\sum_{x_{r-1}\in Y\setminus\{x_1,\ldots,x_{r-2}\}}\ind\{x_1\in A_1,\ldots,x_{r-1}\in A_{r-1}\}
\sum_{x_{r}\in Y}\ind\{x_r\in A_r\}\\
&\qquad =Y^{r-1}_{\neq}(A_1\times\ldots\times A_{r-1})\, Y(A_r)\,,
\end{align*}
since $A_r$ is disjoint from $A_1\cup\ldots\cup A_{r-1}$. Using the fact that $Y^{r-1}_{\neq}(A_1\times\ldots\times A_{r-1})$ and $Y(A_r)$ are 
independent by construction, we obtain that
\begin{align*}
\EE Y^r_{\neq}(A_1\times\ldots\times A_r)&=\EE Y^{r-1}_{\neq}(A_1\times\ldots\times A_{r-1})\, \EE Y(A_r)\\
&=\Big(\prod_{i=1}^{r-1}\EE Y(A_i)\Big)\, \EE Y(A_r)\\
&=\prod_{i=1}^{r}\EE Y(A_i)\,,
\end{align*}
which is \eqref{sr}.

\medskip

\textit{Case II.} If $l(r)\in\{l(1),\ldots,l(r-1)\}$, then, by symmetry, we can assume that $l(i)=l(r)$ for $i\in \{s,\ldots,r\}$ and $l(i)\neq l(r)$ for $i\in \{1,\ldots,s-1\}$, for some $s\in \{1,\ldots,r-1\}$. By the same reasoning as above, it follows that 
\begin{equation}\label{eqstern}
Y^r_{\neq}(A_1\times\ldots\times A_r)=Y^{s-1}_{\neq}(A_1\times\ldots\times A_{s-1})\, Y^{r-s+1}_{\neq}(A_s\times\ldots\times A_r)
\end{equation}
and that $Y^{s-1}_{\neq}(A_1\times\ldots\times A_{s-1})$ and $Y^{r-s+1}_{\neq}(A_s\times\ldots\times A_r)$ are independent. If $s\ge 2$, then the induction hypothesis applies to both factors on the right-hand side of \eqref{eqstern}, since $s-1\le r-1$ and $r-s+1\le r-1$, which completes the induction step. If otherwise $s=1$, then all the sets $A_1,\ldots,A_r$ are subsets of the same cube, $Q_1$ say. Writing $[m]:=\{1,\ldots,m\}$, for $m\in\NN$, and $[0]:=\emptyset$, and recalling that $N_1$ has the same distribution as $N$, we conclude that 
\begin{align*}
&\EE Y^r_{\neq}(A_1\times \ldots\times  A_r)\\
&\qquad = \sum_{m=0}^\infty\PP(N=m)\sum_{i_1\in[m]}\sum_{i_2\in [m]\setminus\{i_1\}}\ldots\sum_{i_r\in  [m]\setminus\{i_1,\ldots,i_{r-1}\}}
\PP(X^1_{i_1}\in A_1,\ldots,X^1_{i_r}\in A_r)\\
&\qquad = \sum_{m=0}^\infty m(m-1)\cdots (m-r+1) \PP(N=m)\prod_{i=1}^r\PP(X^1_i\in A_i)\\
&\qquad = \EE[N(N-1)\cdots (N-r+1)] \prod_{i=1}^r\PP(X^1_i\in A_i)\allowdisplaybreaks\\
&\qquad = (\EE N)^r \,\prod_{i=1}^r\PP(X^1_i\in A_i)= \prod_{i=1}^r\left[\EE(N)\PP(X^1_i\in A_i)\right]\\
&\qquad = \prod_{i=1}^r\EE Y( A_i)\,,
\end{align*}
which is \eqref{sr}. This completes the induction argument in the second case. 

From the preceding argument we also see that 
$$\EE Y^{k+1}_{\neq}(A_1\times\ldots\times A_{k+1})=0$$ 
if $A_1,\ldots,A_{k+1}$  are contained in the same cube. Hence, $Y$ is not of type $(S_{k+1})$ and consequently, $Y$ cannot be a Poisson process. 
\end{proof}

We now construct, for each integer $\ka\geq 2$, a random variable $N$, which satisfies the assumptions of Proposition \ref{thm:App1}, where for simplicity we assume that $\EE N=1$. This can be seen as the Poissonian analogue to a result in \cite{EllisNewman}, where non-Gaussian random variables were investigated, whose moments up to some fixed order coincide with those of a Gaussian random variable.

\begin{proposition}\label{thm:App2}
For each integer $\ka\geq 2$ there exists a random variable $N$ whose distribution has support $\{0,1,\ldots,\ka-2,\ka\}$ and is such that
\begin{equation*}
\EE\left[N(N-1)\cdots(N-m+1)\right]=\left(\EE N\right)^m=1, \qquad\text{for all}\ m\in\NN.
\end{equation*}
\end{proposition}

\begin{proof}
We construct such a random variable recursively and start with the case $\ka=2$. To indicate the dependence of $N$ on $\ka$, we write $N^{(\ka)}$ for the moment. We define $N^{(2)}$ by 
$$\PP(N^{(2)}=0):=\PP(N^{(2)}=2):=\frac{1}{2}\,,$$ which ensures that $\EE(N^{(2)})=\EE(N^{(2)}(N^{(2)}-1))=1$. If the distribution of $N^{(m)}$ with the required properties has been constructed for $m=1,\ldots,\ka-1$, we proceed to define 
\begin{align*}
\PP(N^{(\ka)}=i)&:=\frac{1}{i}\,\PP(N^{(\ka-1)}=i-1)>0\,,\qquad i\in \{1,\ldots,\ka-2,\ka\}\,,\\
\PP(N^{(\ka)}=0)&:=1-\sum_{i\in\{1,\ldots,\ka-2,\ka\}}\PP(N^{(\ka)}=i)\,.
\end{align*}
That the support of the distribution of $N^{(\ka)}$ is indeed concentrated on $\{0,1,\ldots,\ka-2,\ka\}$ can be seen from 
\begin{align*}
1-\sum_{i\in\{1,\ldots,\ka-2,\ka\}}\PP(N^{(\ka)}=i)
& = 1-\sum_{i\in \{1,\ldots,\ka-2,\ka\}}\frac{1}{i}\PP(N^{(\ka-1)}=i-1)\\
& = 1-\PP(N^{(\ka-1)}=0)-\sum_{i\in \{2,\ldots,\ka-2,\ka\}}\frac{1}{i}\PP(N^{(\ka-1)}=i-1)\\
& = \sum_{i\in \{2,\ldots,\ka-2,\ka\}}\left(1-\frac{1}{i}\right)\PP(N^{(\ka-1)}=i-1)> 0\,.
\end{align*}
Here we use that, by induction, the support of the distribution of $N^{(\ka-1)}$ is equal to $\{0,\ldots,\ka-3,\ka-1\}$. 
Furthermore, for $\ka\ge 3$ we have 
$$
\EE  N^{(\ka)}= \sum_{i\in \{0,1,\ldots,\ka-2,\ka\}}i\,\PP(N^{(\ka)}=i)\quad=\sum_{j\in \{0,1,\ldots,\ka-3,\ka-1\}}\PP(N^{(\ka-1)}=j)=1.
$$
For $2\le r\le \ka$, we obtain
\begin{align*}
&\EE\big[N^{(\ka)}(N^{(\ka)}-1)\cdots(N^{(\ka)}-r+1)\big]\\
&\qquad = \sum_{i\in\{r,\ldots,\ka-2,\ka\}}i(i-1)\cdots(i-r+1)\,\PP(N^{(\ka)}=i)\\
&\qquad = \sum_{i\in\{r,\ldots,\ka-2,\ka\}}(i-1)\cdots((i-1)-(r-1)+1)\,\PP(N^{(\ka-1)}=i-1)\\
&\qquad = \sum_{j\in\{r-1,\ldots,\ka-3,\ka-1\}}j\cdots(j-(r-1)+1)\,\PP(N^{(\ka-1)}=j)\\
&\qquad = \EE\big[N^{(\ka-1)}\cdots(N^{(\ka-1)}-(r-1)+1)\big]=1,
\end{align*}
since $1\le r-1\le \ka-1$ and by the properties of the distribution of $N^{(\ka-1)}$.
\end{proof}

\begin{remark}\label{rem:ExApp}\rm 
To illustrate our construction, let us state the distributions of $N^{(3)}$ and $N^{(4)}$ explicitly: 
\begin{align*}
&\PP(N^{(3)}=0)=\frac{1}{3}\,,\qquad\PP(N^{(3)}=1)=\frac{1}{2}\,,\qquad\PP(N^{(3)}=3)=\frac{1}{6}\,,\\
&\PP(N^{(4)}=0)=\frac{3}{8}\,,\qquad\PP(N^{(4)}=1)=\frac{1}{3}\,,\qquad\PP(N^{(4)}=2)=\frac{1}{4}\,,\qquad\PP(N^{(4)}=4)=\frac{1}{24}\,.
\end{align*}
\end{remark}

\medskip

Before we turn to $k$-flat processes in $\RR^n$, we need some preparations. For $d\in\{1,\ldots,n-1\}$ and a (fixed) subspace $E_0\in G(n,n-d)$, we define 
$G_0(n,d):=\{U\in G(n,d):U+E_0=\RR^n\}$ and $A_0(n,d):=\{E\in A(n,d):L(E)\in G_0(n,d)\}$. 
By Lemma 13.2.1 in \cite{SW}, we have $\nu_d(G(n,d)\setminus G_0(n,d))=0$. We write $\mu_d^*$ for the restriction of $\mu_d$ to $ A_0(n,d)$. The map $T:A_0(n,d)\to E_0$, $E\mapsto T(E)$, with $\{T(E)\}=E\cap E_0$, assigning to $E\in A_0(n,d)$ the unique intersection point with $E_0$, is continuous and hence measurable. We shall now compute the image measure of $\mu_d^*$ under $T$. For Borel sets $B\subset E_0$, we have
\begin{align*}
\left(T \mu_d^*\right)(B)&=\int\limits_{A_0(n,d)}\ind\{T(E)\in B\}\, \mu_d^*(\dint E)\\
&=\int\limits_{G_0(n,d)}\int\limits_{L^\perp}\ind\{(L+x)\cap E_0\cap B\neq \emptyset\}\, \ell_{L^\perp}(\dint x)\, \nu_d(\dint L)\\
&=\int\limits_{G_0(n,d)} \ell_{L^\perp}(B\vert L^\perp)\, \nu_d(\dint L)\\
&=\ell_{E_0}(B)\int\limits_{G_0(n,d)}[E_0,L]\, \nu_d(\dint L)\\
&=a(n,d)\,\ell_{E_0}(B)\,,
\end{align*}
with $a(n,d):=c(n,n-d,d)$ given by \eqref{eq:Defcnk1k2}. Consequently, the image measure of $\mu_d^*$ under $T$ is equal to $a(n,d)\ell_{E_0}$. This shows that, for fixed Borel sets $A\subset G(n,d)$ and $B\subset E_0$ with $\ell_{E_0}(B)\in(0,\infty)$,
\begin{equation}\label{QQ0}
\QQ_0(A):=\frac{1}{a(n,d)\,\ell_{E_0}(B)}\int\limits_{A_0(n,d)}\ind\{E-T(E)\in A\}\ind\{T(E)\in B\}\, \mu_d^*(\dint E)
\end{equation}
yields a probability measure on $G(n,d)$. The definition of $\QQ_0$ is independent of the particular choice of $B$, since for fixed $A$, the integral on the right-hand side 
of \eqref{QQ0} defines a locally finite and translation invariant measure in $E_0$ with respect to $B$. Thus, we conclude that
$$
\int\limits_{A_0(n,d)}g(E-T(E),T(E))\, \mu_d(\dint E)=a(n,d)\int\limits_{G(n,d)\times E_0}g(U,x)\, (\QQ_0\otimes \ell_{E_0})(\dint(U,x))
$$
for all measurable functions $g:G(n,d)\times E_0\to [0,\infty]$. Hence, we get 
\begin{equation}\label{decomposition}
\int\limits_{A_0(n,d)}f(E)\, \mu_d(\dint E)=a(n,d) \int\limits_{E_0} \int\limits_{G(n,d)} f(U+x)\, \QQ_0(\dint U)\,\ell_{E_0}(\dint x).
\end{equation}
for all measurable functions $f:A_0(n,d)\to [0,\infty]$. Since the complement of $A_0(n,d)$ in $A(n,d)$ is a set of $\mu_d$-measure zero, we can replace $A_0(n,d)$ by $A(n,d)$ in \eqref{decomposition}. The measure $\QQ_0$, which will serve as a directional distribution in the following, arose from the invariant measure $\mu_d$ by a Palm-type construction with respect to $E_0$. It can be interpreted as the conditional distribution of a random $d$-flat in general position to $E_0$ given that the intersection point with $E_0$ is at the origin.

We now turn to $k$-flat processes. For this, let $k\in\{1,\ldots,n-1\}$ be fixed and put $d:=n-k$. Let $E_0\in G(n,d)$ be an arbitrary subspace, which is identified with $\RR^d$, and,  for some fixed integer $r\geq 2$, let $Y$ be the point process in $E_0$ described above. Furthermore, let $\QQ_0$ be the  probability measure on $G(n,k)$ defined at \eqref{QQ0}. To each point $x\in Y$ we attach an independent random subspace $L\in G(n,k)$ with distribution $\QQ_0$ (independently of $x$). In other words, let ${Y}_{\QQ_0}$ be obtained from $Y$ by independent $\QQ_0$-marking, see \cite[Chapter 3.5]{SW}. Then we define 
\begin{equation}\label{eq:XflatDef}
X:=\sum_{(x,L)\in{Y}_{\QQ_0}}\delta_{L+x}\,.
\end{equation}
As it will follow from the proposition below, the random measure $X$ is locally finite, hence a $k$-flat process in $\RR^n$. 
Since the random variable $N$ from Proposition \ref{thm:App2} has finite support and since $\QQ_0$ is concentrated on $G_0(n,d)$, it follows that $X$ is not a Poisson process. 
We argue now that the $(S_r)$-property of $Y$ carries over to  $X$. Moreover, we show that the intensity measure of $X$ is translation invariant, implying that $X$ is a weakly stationary $k$-flat process in $\RR^n$.

\begin{proposition}
If the point process $Y$ is of type $(S_r)$ for some integer $r\geq 2$, then the $k$-flat process $X$ is of type $(S_r)$. In particular, if the construction of $Y$ is based on random variables from Proposition \ref{thm:App2}, then $X$ is a weakly stationary $k$-flat process of type $(S_r)$, which is not a Poisson process.
\end{proposition}

\begin{proof}
The intensity measure of $X$ is given by
\begin{equation}\label{eq:IntMeasXflat}
\EE X = \int\limits_{E_0}\int\limits_{G(n,d)}\delta_{L+x}\, \QQ_0(\dint L)\, \ell_{E_0}(\dint x)
= \frac{1}{a(n,d)}\int\limits_{A(n,d)}\delta_E\, \mu_d(\dint E)= \frac{1}{a(n,d)}\mu_d\,,
\end{equation}
where \eqref{decomposition} was used and the fact that ${Y}_{\QQ_0}$ is obtained from $Y$ by independent $\QQ_0$-marking. Moreover, 
$X$ is not a Poisson process. In fact, since $Y([0,1]^d\cap E_0)$ is deterministically bounded and $\QQ_0$ is concentrated on $G_0(n,d)$, it 
follows that the number of flats in $X$ hitting $[0,1]^d\cap E_0$ is not Poisson distributed.

The remaining assertion follows from
\begin{align*}
\EE X^r_{\neq}&=\int\limits_{(E_0)^r}\int\limits_{G(n,d)^r}\delta_{(L_1+x_1,\ldots,L_r+x_r)}\, (\QQ_0)^r(\dint (L_1,\ldots,L_r))
\, \left(\EE Y\right)^r_{\neq}(\dint(x_1,\ldots,x_r))\\
&=\int\limits_{(E_0)^r}\int\limits_{G(n,d)^r}\delta_{(L_1+x_1,\ldots,L_r+x_r)}\, (\QQ_0)^r(\dint (L_1,\ldots,L_r))
\, \left(\EE Y\right)^r(\dint(x_1,\ldots,x_r))\\
&=(\EE X)^r\,,
\end{align*}
where we first used the fact that $X$ arises from independent $\QQ_0$-marking of $Y$, then the assumption that $Y$ is of type $(S_r)$, and finally Fubini's theorem and the special case $r=1$ of the first equality.
\end{proof}

Let us summarize our results. First, we have constructed for any integers $d\geq 1$ and $r\geq 2$ a weakly stationary point process $Y$ of type $(S_r)$ in $\RR^d$, which is not a Poisson process. Based on $Y$, we have then constructed for all integers $n\geq 2$, $k\in\{1,\ldots,n-1\}$ and $r\geq 2$ a weakly stationary $k$-flat process $X$ in $\RR^n$ of type $(S_r)$, which is not a Poisson process. In this context, the following natural question arises: Is it possible to replace weak stationarity by stationarity? We will provide a positive answer in the point process case ($k=0$) and pose it as an open problem to construct, for $k\geq 1$, a stationary $k$-flat process of type $(S_r)$, which is not a Poisson process.   In contrast to the point processes case, the results in \cite{Kallenberg80} show that a stationary $k$-flat process satisfying a regularity condition is already a Cox process (i.e., a doubly stochastic Poisson process); see, for instance, Theorem 4.2 and Corollary 5.2 in [23]. Moreover, an arbitrary Cox process $Z$ with directing random measure $M$, which is of type $(S_r)$ for some $r\ge 2$,  satisfies $\EE M^r=\EE Z^{r}_{\neq}=(\EE Z)^r=(\EE M)^r$, where the first equality is Theorem 4.2 (i) in \cite{Kallenberg2011}. Hence, $M(B)$ is $\PP$-almost surely constant, for every measurable set $B$. Thus, $M$ is deterministic and therefore $Z$ is a  Poisson process. These facts provide considerable restrictions to examples of stationary $k$-flat processes of type $(S_r)$, which are not Poisson processes.

Let us now return to the point processes case. As noticed above, the point process $Y$ defined in \eqref{eq:defPPX} is not stationary, if the random variable $N$ has finite support. For this reason we follow the strategy already used in \cite{BaddeleySilverman} and perturb $Y$ by a random shift to enforce stationarity. Formally, let $\xi$ be a uniform random point in $[0,1]^d$, which is independent of all other random variables $\{X^i_j\}_{i,j\geq 1}$ and $\{N_i\}_{i\geq 1}$, and define $\widetilde{Y}:=Y+\xi$. It is not difficult to see that $\widetilde{Y}$ is indeed a stationary point process, which is not a Poisson process as long as $N$ is not Poisson distributed. We show now that the perturbed point process $\widetilde{Y}$ inherits the $(S_r)$-property from $Y$. 

\begin{proposition}\label{thm:App4}
Let $\ka\geq 2$ be an integer, and let $N$ and $Y$ be defined as in Proposition \ref{thm:App1}. Then $\widetilde{Y}$ is a stationary process of type $(S_r)$ for $r=2,\ldots,\ka$, but not Poisson.
\end{proposition}

\begin{proof}
We may assume that $\EE N =1$. Then,  $Y$ has intensity one and the intensity measure of $Y$ is the Lebesgue measure on $\RR^d$. 
Let $r\in \{2,\ldots,k\}$ and fix Borel sets $A_1,\ldots,A_r\subset\RR^d$. Then,
\begin{align*}
\EE\widetilde{Y}_{\neq}^r(A_1\times\ldots\times A_r)&=\EE\sum_{(x_1,\ldots,x_r)\in(Y+\xi)_{\neq}^r}\ind\{(x_1,\ldots,x_r)\in A_1\times\ldots\times A_r\}\\
&=\EE\sum_{(x_1,\ldots,x_r)\in Y_{\neq}^r}\ind\{(x_1+\xi,\ldots,x_r+\xi)\in A_1\times\ldots\times A_r\}\\
&=\int\limits_{(\RR^d)^r}\int\limits_{[0,1]^d}\ind\{x_1+z\in A_1,\ldots,x_r+z \in A_r\}\,\ell_d(\dint z)\, \EE Y^r_{\neq}(\dint(x_1,\ldots,x_r))\,.
\end{align*}
Since $Y$ is of type $(S_r)$, we conclude that
\begin{align*}
\EE\widetilde{Y}_{\neq}^r(A_1\times\ldots\times A_r)&=\int\limits_{\RR^d}\ldots\int\limits_{\RR^d} \ell_d\left((A_1-x_1)\cap\ldots\cap(A_r-x_r)\cap[0,1]^d\right)\, \ell_d(\dint x_1)\ldots\ell_d(\dint x_r)\\
&=\ell_d(A_1)\cdots \ell_d(A_r)\,,
\end{align*}
by Fubini's theorem. On the other hand, since $\widetilde{Y}$ is stationary, we get
$$
\int\limits_{(\RR^d)^r}\ind\{(x_1,\ldots,x_r)\in A_1\times\ldots\times A_r\}\,\big(\EE\widetilde{Y}\big)^r(\dint (x_1,\ldots,x_r))
=\ell_d(A_1)\cdots \ell_d(A_r)\,,
$$
which shows that $\EE\widetilde{Y}_{\neq}^r=(\EE\widetilde{Y})^r$. 
\end{proof}

\subsection*{Acknowledgement}
The authors are grateful to Lothar Heinrich for discussions on the topic of the Appendix.

DH and WW have been supported by Deutsche Forschungsgemeinschaft via the Research Group ``Geometry and Physics of Spatial Random Systems''. CT has been supported by Deutsche Forschungsgemeinschaft via SFB/TR 12 ``Symmetries and Universality in Mesoscopic Systems''.


\end{document}